\documentclass[11pt, a4paper]{amsart}
\usepackage[utf8]{inputenc}

\usepackage{amsmath, amsfonts, amssymb, amsthm,mathtools,enumerate}
\allowdisplaybreaks
\usepackage{graphicx}
\usepackage[inline]{enumitem}
\usepackage{hyperref}
\usepackage{stmaryrd} 
\usepackage{multicol}

\usepackage{diagbox}
\usepackage{float}

\theoremstyle{definition}
\newtheorem{definition}{Definition}[section]
\newtheorem{remark}[definition]{Remark}
\newtheorem{example}[definition]{Example}
\theoremstyle{plain}
\newtheorem{proposition}[definition]{Proposition}
\newtheorem{propdef}[definition]{Proposition-Definition}
\newtheorem{lemma}[definition]{Lemma}
\newtheorem{theorem}[definition]{Theorem}

\newtheorem{corollary}[definition]{Corollary}
\newtheorem{introtheorem}{Theorem}

\usepackage[
    a4paper,
    top=2.5cm,
    bottom=2.5cm,
    left=2.5cm,
    right=2.5cm,
    includeheadfoot,
    headheight=14pt,
    footskip=1cm,
    marginparwidth=3cm,
    marginparsep=0.5cm
]{geometry}

\newlist{steplist}{enumerate}{1}
\setlist[steplist]{label=\textbf{Step \arabic*}, align=left, wide=1pt}

\newlist{caselist}{enumerate}{1}
\setlist[caselist]{label=\textbf{Case \arabic*}, align=left, wide=1pt}

\newcommand{\Z}{\mathbb{Z}}
\newcommand{\Q}{\mathbb{Q}}
\newcommand{\R}{\mathbb{R}}
\newcommand{\C}{\mathbb{C}}
\newcommand{\A}{\mathbf{A}} 
\newcommand{\AX}{\mathbf{A}\langle X_G \rangle} 
\newcommand{\AXX}{\mathbf{A}\langle\langle X_G \rangle \rangle} 
\newcommand{\AXXGd}{\mathbf{A}\langle\langle X_{G^d} \rangle \rangle} 
\newcommand{\AYY}{\mathbf{A}\langle\langle Y_G \rangle \rangle} 
\newcommand{\h}{\mathfrak{H}} 
\newcommand{\reg}{\mathrm{reg}} 
\newcommand{\ev}{\mathrm{ev}} 
\newcommand{\EDS}{\mathsf{EDS}} 
\newcommand{\EDSD}{\mathsf{EDSD}} 
\newcommand{\DMR}{\mathsf{DMR}} 
\newcommand{\DMRD}{\mathsf{DMRD}} 
\newcommand{\q}{\operatorname{q}} 
\newcommand{\qq}{\mathbf{q}} 
\newcommand{\G}{\mathcal{G}} 
\newcommand{\st}{\ast} 

\newcommand{\Li}{\operatorname{Li}}
\DeclareFontFamily{U}{wncy}{}
\DeclareFontShape{U}{wncy}{m}{n}{<->wncyr10}{}
\DeclareSymbolFont{mcy}{U}{wncy}{m}{n}
\DeclareMathSymbol{\sh}{\mathord}{mcy}{"78} 

\newcommand{\LL}{\mathcal{L}}
\newcommand{\QL}{\Q\langle \LL \rangle}
\newcommand{\ALL}{\A\langle \langle \LL \rangle \rangle}
\newcommand{\qsh}{\ast_\diamond}
\newcommand{\rsh}{\overline{\reg}_\sh}

\hypersetup{pdfborder={0 0 0}} 

\newcommand{\df}{\coloneqq}

\title[Conjecture of Zhao and standard relations]{On a conjecture of Zhao related to standard relations among cyclotomic multiple zeta values}
\author{Henrik Bachmann}
\address{{\scriptsize Graduate School of Mathematics, Nagoya University, Furo-cho, Chikusa-ku, Nagoya, 464-8602, Japan.}}
\email{henrik.bachmann@math.nagoya-u.ac.jp}
\author{Khalef Yaddaden}
\address{{\scriptsize Graduate School of Mathematics, Nagoya University, Furo-cho, Chikusa-ku, Nagoya, 464-8602, Japan.}}
\email{khalef.yaddaden.c8@math.nagoya-u.ac.jp }


\begin{document}
    \begin{abstract}
        We provide a proof of a conjecture by Zhao concerning the structure of certain relations among cyclotomic multiple zeta values in weight two. We formulate this conjecture in a broader algebraic setting in which we give a natural equivalence between two schemes attached to a finite abelian group $G$. In particular, when $G$ is the group of roots of unity, these schemes describe the standard relations among cyclotomic multiple zeta values. 
    \end{abstract}

    \maketitle
	
    {\footnotesize \tableofcontents}
    
    \section*{Introduction}

The purpose of this paper is to give a proof of a conjecture of Zhao and to compare two distinct frameworks describing the algebraic relations among multiple zeta values and more generally multiple polylogarithms at roots of unity.
\emph{Multiple polylogarithms} generalize the classical polylogarithms and are defined for $k_1,\dots,k_r\geq 1$ by 
\begin{equation}\label{eq:defli}
\Li_{(k_1,\dots,k_r)}(z_1,\dots,z_r) = \sum_{n_1>\dots>n_r>0} \frac{z_1^{n_1} \cdots z_r^{n_r}}{n_1^{k_1} \cdots n_r^{k_r}},
\end{equation}
which converge for suitable values of $z_i$. In \eqref{eq:defli} we call $r$ the \emph{depth} and $k_1+\dots+k_r$ the \emph{weight}. 
The special case where $k_1 \geq 2$ and $z_1 = \dots = z_r = 1$, yields \emph{multiple zeta values (MZVs)} $\zeta(k_1,\dots,k_r) = \Li_{(k_1,\dots,k_r)}(1,\dots,1)$. From the perspective of Ihara, Kaneko, and Zagier \cite{IKZ}, conjecturally all algebraic relations of MZVs are a consequence of the extended double shuffle (EDS) relations.  Another special case for the values of $z_i$ in \eqref{eq:defli} is when they are $N$-th roots of unity. The resulting values are sometimes referred to as colored or \emph{cyclotomic multiple zeta values} of level $N$. Besides the extended double shuffle relations these values also satisfy the so-called finite and regularized distribution relations. The main result of this work is the following.

\begin{introtheorem}[{\cite[Conjecture 4.7]{Zha10}}]\label{thm:introzhao} In weight two, all regularized distribution relations are consequences of the extended double shuffle relations, finite distribution relations of weight one and finite distribution relations of depth two.
\end{introtheorem}

We will prove a more general variant of Theorem \ref{thm:introzhao}, which provides a similar statement regarding a generalization of the distribution and double shuffle relations. This generalization will be presented within two distinct algebraic frameworks. Theorem \ref{thm:introzhao} will follow as a consequence of Theorem \ref{thm:zhaoconj} and Proposition \ref{conj_equiv_thm}, established in this broader context. One framework extends the algebraic setup of Ihara, Kaneko, and Zagier (\cite{IKZ}), while the other was introduced by Racinet (\cite{Rac02}). In the following, we will outline these two frameworks and compare them. \\

Following Ihara, Kaneko, and Zagier (\cite{IKZ}), the double shuffle relations can be viewed as a comparison between two product formulas satisfied by MZVs. To describe them, one introduces a space $\h^0$ equipped with two products: $\sh$ (shuffle product) and $\st$ (harmonic product). MZVs are then seen as an algebra homomorphism from $\h^0$ to $\R$ with respect to both products. By extending this map in two natural ways, through regularization, to two algebra homomorphisms from a larger algebra $\h^0 \subset \h^1$ to $\R[T]$, along with an explicit comparison map $\rho: \R[T] \rightarrow \R[T]$ between both regularizations, the extended double shuffle relations are established. In the context of cyclotomic multiple zeta values, one can similarly describe the algebraic relations satisfied by these values (e.g. \cite{AK,Tas21,YZ}) by considering them as algebra homomorphisms of an extension of $\h^0$. 

In this paper, we generalize this algebraic framework by introducing a space $\h^0_G$ associated with a finite abelian group $G$. We then provide a definition of what it means for a $\Q$-linear map $Z_\A$ from $\h^0_G$ to a $\Q$-algebra $\A$ to satisfy the extended double shuffle relations. In the cases where $G$ is trivial or where $G$ is the group of $N$-th roots of unity, this notion reduces to the classical cases of multiple zeta values and their level $N$ analogues, respectively. For a given group $G$ and $\Q$-algebra $\A$ we then introduce the set $\EDS(G)(A)$ of all such maps. \\ 

A somehow dual point of view of the above story is Racinet's (\cite{Rac02}) approach of the study of MZVs through the lens of groupoid-enriched Hopf algebras, emphasizing the role of group actions and the double shuffle relations in this context. In his framework, one considers formal non-commutative power series with coefficients in $\A$, which are, after a certain correction, group-like for two different coproducts. Denote by $\DMR(G)(\A)$ (Definition \ref{def:dmrg}) the set of all such series. This formalism builds upon the notion of Drinfeld associator introduced in \cite{Dri91} which is a generating series of MZVs that satisfies associator relations describing periods of mixed Tate motives. It is shown in \cite{Fur11} that associator relations imply double shuffle relations and it is conjectured to be equivalent. A cyclotomic analogue of Drinfeld associator called cyclotomic associator is studied by \cite{Enr08} and its connection with double shuffle relations was established in \cite{Fur13}. \\

Both the setup introduced by Ihara, Kaneko, and Zagier and the one introduced by Racinet have been used by various authors to describe the relations among multiple polylogarithms or (cyclotomic) multiple zeta values. Their equivalence is in some sense common knowledge among the community but, as far as the authors know, it was not written down precisely somewhere so far, except for the classical case of trivial $G$ (cf. \cite[Proposition 3.1]{ENR}, \cite{Bur23}). 
In this work, we give an explicit comparison for arbitrary $G$ by assigning to a $Z_\A \in \EDS(G)(\A) $ an element $\Phi_{Z_\A \circ \rsh} \in \DMR(G)(\A)$ (Definition \ref{def:phiza}), such that we have the following

\begin{introtheorem}[{Theorem \ref{main_bijection}}]\label{main_bijection_intro}
    The following map is bijective
    \begin{equation}\label{eq:edsdmrintrobij}
        \begin{array}{ccc}
             \EDS(G)(\A) & \longrightarrow & \DMR(G)(\A), \\
             Z_\A & \longmapsto & \Phi_{Z_\A \circ \rsh}.
        \end{array}
    \end{equation}
\end{introtheorem}

Racinet's viewpoint introduces a more general setup then the one of Ihara, Kaneko, and Zagier, accommodating also the so-called distribution relations. For multiple polylogarithms these were explicited by Goncharov \cite{Gon98, Gon01}.

Distribution relations are not a consequence of the extended double shuffle relations and they are trivial in the classical case if $G$ is trivial. Both are part of the so-called \emph{standard relations} among the cyclotomic multiple zeta values (\cite{Zha10}) in the case if $G$ is the group of roots of unity.
The distribution relations were incorporated by Racinet by introducing a subset $\DMRD(G)(\A) \subset \DMR(G)(\A)$, which gives the subset of elements that additionally satisfy these relations. As a counterpart to this, we introduce the subset $ \EDSD(G)(\A) \subset \EDS(G)(\A)$ (Definition \ref{def:edsd}) and show the following. 
\begin{introtheorem}[{Theorem 
    \ref{main_bijection_2}}] \label{main_bijection_2_intro}
    The bijection \eqref{eq:edsdmrintrobij} restricts to a bijection
    \[
        \EDSD(G)(\A) \longrightarrow \DMRD(G)(\A).
    \]
\end{introtheorem}

While the result of Theorem \ref{main_bijection_intro} may be regarded as a well-known fact among experts, it has not yet been explicitly documented in the literature in the general form presented here. Nevertheless, we emphasize that the result in Theorem \ref{main_bijection_2_intro} is novel, as the concept of the space $\EDSD(G)(\A)$ is entirely new and has not been discussed or introduced elsewhere.\\

    \paragraph{\textbf{Acknowledgments}}
    This project was partially supported by JSPS KAKENHI Grants 23K03030 and 23KF0230. The authors would like to thank Hidekazu Furusho and Benjamin Enriquez for fruitful comments on this project. \\

    \paragraph{\textbf{Notation}}
    Throughout this paper, we will use the following notations
    \begin{enumerate}[label=(\roman*)]
        \item $G$ is a finite abelian (multiplicative) group.
        \item $X_G \df \{x_0\} \sqcup \{x_g | g \in G\}$.
        \item $Y_G \df \{ y_{n,g} | (n,g) \in \Z_{>0} \times G\}$.
        \item $\A$ is a commutative $\Q$-algebra with unit.
        \item $\LL^\ast$ is the monoid generated by a set $\LL$.
    \end{enumerate}

    \section{Double shuffle relations}

In this section, we review two fundamental formalisms used in the algebraic setup of double shuffle relations. We recall Racinet's formalism as presented in \cite{Rac02} encoding double shuffle relations into a scheme $\mathsf{DMR}(G)$ attached to a group $G$. Another formalism was developed by Ihara, Kaneko and Zagier \cite{IKZ} for trivial $G$ and Arakawa, Kaneko \cite{AK} for general $G$. Following Racinet's approach, we encode this formalism into a scheme denoted as $\mathsf{EDS}(G)$. The section concludes by proving that the schemes $\mathsf{DMR}(G)$ and $\mathsf{EDS}(G)$ are isomorphic, thus establishing formally the equivalence between the two formalisms.      

\subsection{The scheme \texorpdfstring{$\EDS$}{EDS} of extended double shuffle relations}

Denote by $\h_G \df \mathbb{Q}\langle X_G \rangle$ the free non-commutative associative polynomial $\Q$-algebra with unit over the alphabet $X_G$. Set
\[
    \h^1_G \df \Q + \sum_{g \in G} \h_G x_g \quad \text{ and }\quad \h^0_G \df \Q + \sum_{g \in G} x_0 \h_G x_g + \sum_{g \in G \neq \{1\}} x_g \h^1_G.  
\]
One checks that we have the following inclusion of $\Q$-algebras
\[
    \h^0_G \subset \h^1_G \subset \h_G.
\]

\noindent For $(n,g) \in \Z_{>0} \times G$, the assignment $y_{n,g} \mapsto x_{0}^{n-1} x_{g}$ establishes a $\Q$-algebra isomorphism $\mathbb{Q}\langle Y_G \rangle \simeq \h^1_G$. We shall identify these two algebras thanks to this isomorphism.

\begin{definition}
    We define the \emph{shuffle product} $\sh$ on $\h_G$ inductively by
    \begin{equation*}
        \begin{cases}
            1 \ \sh \ w = w \ \sh \ 1 = w & \text{for } w \text{ a word in } \h_G \\
            u w_{1} \ \sh \ v w_{2} = u(w_{1} \ \sh \ v w_{2}) + v(u w_{1} \ \sh \ w_{2}) & \text{for } w_1, w_2 \text{ words in } \h_G \text{ and } u, v \in X_G,
        \end{cases}
    \end{equation*}
    and then extending by $\mathbb{Q}$-bilinearity.
\end{definition}
\noindent The product $\sh$ gives $\h_G$ the structure of a commutative $\Q$-algebra, which we denote by $\h_{G, \sh}$. The subspaces $\h^1_G$ and $\h^0_G$ then become subalgebras of $\h_{G, \sh}$ denoted $\h_{G, \sh}^{1}$ and $\h_{G, \sh}^{0}$ respectively.

\begin{definition}
    We define the \emph{harmonic product} $\st$ on $\h^1_G$ inductively by
\begin{equation*}
    \begin{cases}
        1\st w=w\st 1=w \\
        y_{n_1, g_1} w_{1} \st y_{n_2, g_2} w_{2} = 
        y_{n_1, g_1}( w_{1} \st y_{n_2 ,g_2} w_{2}) +y_{n_2, g_2}( y_{n_1, g_1} w_{1} \st w_{2}) +y_{n_1 +n_2 ,\ g_1 g_2}(w_{1} \st w_{2}),
    \end{cases}
\end{equation*}
for words $w, w_1, w_2$ in $\h^1_G$, $(n_1 ,g_1), (n_2 ,g_2) \in \Z_{>0} \times G$; and then extending by $\Q$-bilinearity.
\end{definition}
\noindent One checks that $\h_{G, \st}^{1} \df \left(\h^1_G, \st \right)$ is a commutative $\Q$-algebra and that $\h_{G, \st}^{0} \df \left(\h^0_G, \st \right)$ is a subalgebra of $\h_{G, \st}^{1}$.

Let $\q_G$ be the $\Q$-linear automorphism of $\h_G$ given by (\cite[§2.2.7]{Rac02})
\[
    \mbox{\small $\q_G(x_0^{n_1-1}x_{g_1}x_0^{n_2-1}x_{g_2} \cdots x_0^{n_r-1}x_{g_r} x_0^{n_{r+1}-1}) = x_0^{n_1-1}x_{g_1}x_0^{n_2-1}x_{g_2g_1^{-1}} \cdots x_0^{n_r-1}x_{g_rg_{r-1}^{-1}}x_0^{n_{r+1}-1}$},
\]
for $r, n_1, \dots, n_{r+1} \in \Z_{>0}$ and $g_1, \dots, g_r \in G$. Its reciprocal is given by
\[
    \mbox{\small $x_0^{n_1-1}x_{g_1}x_0^{n_2-1}x_{g_2} \cdots x_0^{n_r-1}x_{g_r} x_0^{n_{r+1}-1} \mapsto x_0^{n_1-1}x_{g_1}x_0^{n_2-1}x_{g_1g_2} \cdots x_0^{n_r-1}x_{g_1 \cdots g_r}x_0^{n_{r+1}-1}$}.
\]
One checks that $\q_G$ restricts to a $\Q$-linear automorphism of $\h^1_G$ which we denote $\q_G$ as well and is given by
\[
    \q_G( y_{n_1,g_1} y_{n_2,g_2}  \cdots y_{n_r,g_r})  = y_{n_1,g_1} y_{n_2,g_2 g_1^{-1}}  \cdots y_{n_r,g_rg_{r-1}^{-1}}. 
\]
Moreover, $\q_G$ also restricts to an $\A$-module automorphism of $\h^0_G$ also denoted $\q_G$.

\begin{definition}
    A $\Q$-linear map $Z_{\A} :\h^{0} \to \A$ is said to have the \emph{finite double shuffle} property if it is an algebra homomorphism for $\sh$ and if $Z_{\A} \circ \q_G^{-1}$  is an algebra homomorphism for $\st$. i.e.
    \begin{align*}
        Z_{\A}(w_{1}) Z_{\A}(w_{2}) & = Z_{\A}(w_{1} \ \sh \ w_{2}), \\
        Z_{\A}(\q_G^{-1}(w_{1})) Z_{\A}(\q_G^{-1}(w_{2})) & = Z_{\A}(\q_G^{-1}(w_{1} \st w_{2})),
    \end{align*}
    for all $ w_{1} ,w_{2} \in \h^{0}$.
\end{definition}

\begin{example}\label{ex:polylog}
    In the classical case (\cite{AK},\cite{IKZ}), one takes $\A=\C$ and for a fixed $N\geq 1$ one considers the group of $N$-roots of unity $G=\mu_N$. Let us consider the $\Q$-linear map $Z_\C : \h^0_{\mu_N} \to \C$ given by $1 \mapsto 1$ and for $k>0$,
    \[
        u_1 \cdots u_k \longmapsto \int_0^1 \Omega_{u_1} \cdots \Omega_{u_k},
    \]
    where $\Omega_{x_0} = \frac{dt}{t}$ and $\Omega_{x_z} = \frac{dt}{z^{-1} - t}$. Here, the iterated integrals are recursively defined by
    \[
        \begin{cases}
            \displaystyle \int_a^b \omega_1 = \int_a^b f_1(t) dt & \\
            \displaystyle \int_a^b \omega_1 \cdots \omega_n = \int_a^b \left(\int_a^t \omega_2 \cdots \omega_n\right) f_1(t) dt & \text{if } n > 1,
        \end{cases}
    \]
    where $\omega_1, \dots, \omega_n$ are complex valued differential $1$-forms defined on a real interval $[a,b]$ such that $\omega_i = f_i(t) dt$ with complex functions $f_1, \dots, f_n$. It is known that (see, for example, \cite[Theorem 2.1]{Gon98}) 
    \begin{align}
        \Li_{(k_1,\dots,k_r)}(z_1,\dots,z_r) & = \int_0^1 \Omega_{x_0}^{k_1-1} \Omega_{x_{z_1}} \Omega_{x_0}^{k_2-1} \Omega_{x_{z_1 z_2}} \cdots \Omega_{x_0}^{k_r-1} \Omega_{x_{z_1 \cdots z_r}} \notag \\
        & = Z_\C \circ \q_{\mu_N}^{-1}(x_0^{k_1-1}x_{z_1}x_0^{k_2-1}x_{z_2} \cdots x_0^{k_r-1}x_{z_r}). \label{polylog_Zcircq}
    \end{align}
    Therefore, for instance, the harmonic product
    \[
        y_{1,z_1} \st y_{1,z_2} =  y_{1,z_1} y_{1,z_2} +  y_{1,z_2} y_{1,z_1} +  y_{2,z_1 z_2}
    \]
    corresponds to the identity
    \[
        \Li_1(z_1) \Li_1(z_2) = \Li_{(1,1)}(z_1,z_2) + \Li_{(1,1)}(z_2,z_1) + \Li_2(z_1 z_2).   
    \]
    The shuffle product
    \[
        x_{z_1} \, \sh \, x_{z_2} = x_{z_1} x_{z_2} + x_{z_2} \sh x_{z_1}  
    \]
    corresponds to the identity
    \[
        \Li_1(z_1) \Li_1(z_2) = \Li_{(1, 1)}(z_1, z_1^{-1} z_2) + \Li_{(1, 1)}(z_2, z_1z_2^{-1}).   
    \]
    Finite double shuffle relations is then expressed by the identity
    \[
        \Li_{(1,1)}(z_1,z_2) + \Li_{(1,1)}(z_2,z_1) + \Li_2(z_1 z_2) = \Li_{(1, 1)}(z_1, z_1^{-1} z_2) + \Li_{(1, 1)}(z_2, z_1z_2^{-1}).
    \]
\end{example}

\begin{definition}
    Define the $\Q$-algebra homomorphism $\rsh^T : \h_{G, \sh} \to \h^0_{G, \sh}[T]$ such that it is the identity on $\h^0$ and maps $x_0$ to $0$ and $x_1$ to $T$. \newline
    Define the $\Q$-algebra homomorphism $\reg_\sh^T : \h^1_{G, \sh} \to \h^0_{G, \sh}[T]$ to be the restriction of $\rsh^T$ to $\h^1_G$ (see \cite[§3]{Zha10} and \cite[§3]{IKZ} for $G=\{1\}$).  
\end{definition}

\begin{lemma} \label{lem:barregformula}
    For $w = x_g w^\prime \in \h_G$, with $g \in \{0\} \sqcup G \smallsetminus \{1\}$, we have the following equality in $\h^0_{G, \sh}[T]$
    \[
        \overline{\reg}_{\sh}^T\left(\frac{1}{1-x_1u} w\right) = x_g \left(\frac{1}{1+x_1u}  \sh \,\widetilde{\reg}_{\sh}(w^\prime)\right) \exp(Tu),
    \]
    where $u$ is a formal parameter and $\widetilde{\reg}_\sh: \h_{G,\sh} \rightarrow \h^1_{G,\sh}$ is the $\Q$-algebra homomorphism such that it is the identity on $\h^1_G$ and maps $x_0$ to $0$.
\end{lemma}
\begin{proof}
    This is a generalization of \cite[Proposition 8]{IKZ}. Its proof, and the one of \cite[Proposition 7]{IKZ} which is used therein, can be adapted to our setup without any major change.
\end{proof}

\begin{corollary} \label{cor:regformula}
    For $w = x_g w^\prime \in \h^0_G$ we have
    \[
        \reg_{\sh}^T\left(\frac{1}{1-x_1u} w\right) = x_g \left(\frac{1}{1+x_1u}  \sh w^\prime\right) \exp(Tu).
    \]
\end{corollary}
\begin{proof}
    This is an immediate consequence of Lemma \ref{lem:barregformula}.
\end{proof}

\begin{lemma}\label{lem:extendalghomh0}
    Let $Z_{\A} :\h^{0}_{G, \sh} \to \A$ be a $\Q$-algebra homomorphism. Then it extends to a unique $\Q$-algebra homomorphism
    \[
        \overline{Z}_{\A}^{\sh} :\h_{G, \sh} \to \A[T]
    \]
    which agrees with $Z_{\A}$ on $\h^0_G$ and maps $x_0$ to $0$ and $x_1$ to $T$.
\end{lemma}
\begin{proof}
    It follows by considering the compositions
    \begin{equation}
        \label{compositionZsh}
        \h_{G, \sh} \xrightarrow{\rsh^{T}} \h^0_{G, \sh}[T] \simeq \h^0_{G, \sh} \otimes \mathbb{Q}[T] \xrightarrow{Z_{\A}  \otimes  \mathrm{id}} \A \otimes \mathbb{Q}[T] \simeq \A[T].
    \end{equation}
\end{proof}

\begin{corollary}\label{cor:extendalghomh0}
    Let $Z_{\A} :\h^0_{G, \sh} \to \A$ be a $\Q$-algebra homomorphism. Then it extends to a unique $\Q$-algebra homomorphism
    \[
        Z_{\A}^{\sh} :\h^1_{G,\sh} \to \A[T]
    \]
    which agrees with $Z_{\A}$ on $\h^0_G$ and maps $x_1$ to $T$.    
\end{corollary}
\begin{proof}
    It follows by restriction of $\overline{Z}_A^\sh$ from Lemma \ref{lem:extendalghomh0} to $\h^1_G$.
\end{proof}

\begin{corollary}
    \label{Zreg}
    Let $Z_{\A} : \h^{0}_{G, \sh} \to \A$ be a $\Q$-algebra homomorphism. We have (equality of $\Q$-algebra morphisms $\h_{G, \sh} \to \A$)
    \[
        Z_{\A} \circ \rsh = \ev_0^\A \circ \overline{Z}_\A^\sh.
    \]
\end{corollary}
\begin{proof}
    It follows from composition \eqref{compositionZsh}.
\end{proof}

We define the $\A$-module automorphism $\rho_{Z_\A}$ of $\A[T]$ such that
\[
    \rho_{Z_\A}(\exp(Tu)) = \Gamma_{Z_{\A}}(u) \exp(Tu),
\]
where $\Gamma_{Z_{\A}}(u)$ is the power series defined by
\[
    \Gamma_{Z_{\A}}(u) \df \exp\left(\sum_{n \geq 2} \frac{(-1)^n}{n} Z_{\A}(x_0^{n-1} x_1) u^n \right).
\]
\begin{definition}
    We define $\EDS(G)(\A)$ to be the set of elements $Z_\A \in \mathrm{Hom}_{\Q}(\h^0_G, \A)$ such that
    \begin{enumerate}[label=(\alph*), leftmargin=*]
        \item $Z_\A :\h^0_{G, \sh} \to \A$ is an algebra homomorphism;
        \item $\rho_{Z_\A}^{-1} \circ Z_{\A}^{\sh} \circ \q_G^{-1} : \h^1_{G, \st} \to \A[T]$ is an algebra homomorphism.
    \end{enumerate}
    An element of $\EDS(G)(A)$ is said to have the \emph{extended double shuffle} property.  
\end{definition}

\begin{remark}
    Assume that a $\Q$-linear map $Z_{\A} :\h_G^0 \to \A$ has the \emph{finite double shuffle} property, that is, it is an algebra homomorphism for $\sh$ and $Z_{\A} \circ \q^{-1}_G$ is an algebra homomorphism for $\st$. Similar as in Corollary \ref{cor:extendalghomh0} we can also extend the algebra homomorphism $Z_{\A} \circ \q^{-1}_G : \h^0_{G, \st} \to \A$ to an algebra homomorphism $Z_{\A}^{\st} : \h^1_{G, \st} \to \A[T]$. If $Z_{\A}$ has the extended double shuffle property, then we have $Z_{\A}^{\st} = \rho^{-1} \circ Z_{\A}^{\sh}$.
\end{remark}

\begin{example} The $\Q$-linear map $Z_\C : \h^0_{\mu_N} \to \C$ in Example \ref{ex:polylog} has the extended double shuffle property (cf. \cite[Theorem 2.9]{Zha10}) and therefore gives and element in $\EDS(\mu_N)(\C)$.
\end{example}

\subsection{The scheme \texorpdfstring{$\DMR$}{DMR} of double shuffle and regularization relations}
Let $\AXX$ be the free non-commutative associative series $\A$-algebra with unit over the alphabet $X_G$. It is the completion of the graded $\A$-algebra $\A\langle X_G\rangle \simeq \h_G \otimes \A$ where elements of $X_G$ has degree $1$. 
One equips the algebra $\AXX$ with the Hopf algebra coproduct $\widehat{\Delta}_{G, \sh} : \AXX \to {\AXX}^{\hat{\otimes} 2}$, called the \emph{shuffle coproduct}, which is the unique $\A$-algebra homomorphism given by $\widehat{\Delta}_{G, \sh}(x_g) = x_g \otimes 1 + 1 \otimes x_g$, for any $g \in G \sqcup \{0\}$ (\cite[§2.2.3]{Rac02}).
Let then $\G(\AXX)$ be the set of grouplike elements of $\AXX$ for the coproduct $\widehat{\Delta}_{G, \sh}$, i.e. the set
\begin{equation*}
    \G(\AXX) = \{ \Phi \in \AXX^{\times} \, | \, \widehat{\Delta}_{G, \sh}(\Phi) = \Phi \otimes \Phi \},
\end{equation*}
where $\AXX^{\times}$ denotes the set of invertible elements of $\AXX$. 

Let $\AYY$ be the free non-commutative associative series $\A$-algebra with unit over the alphabet $Y_G$. It is the completion of the graded $\A$-algebra $\A\langle Y_G\rangle \simeq \h^1_G \otimes \A$ where elements $y_{n,g}$ ($(n,g) \in \Z_{>0} \times G$) of $Y_G$ has degree $n$. The algebra $\AYY$ can be viewed as a subalgebra of $\AXX$ thanks to the mapping $y_{n,g} \mapsto x_0^{n-1} x_g$. One equips the algebra $\AYY$ with the Hopf algebra coproduct $\widehat{\Delta}_{G, \st} : \AYY \to \AYY^{\hat{\otimes} 2}$, called the \emph{harmonic coproduct}, which is the unique $\A$-algebra homomorphism such that (\cite[§2.3.1]{Rac02})
\begin{equation*}
    \widehat{\Delta}_{G, \st}(y_{n,g}) = y_{n,g} \otimes 1 + 1 \otimes y_{n,g} + \sum_{\substack{k=1 \\ h \in G}}^{n-1} y_{k,h} \otimes y_{n-k,gh^{-1}},
\end{equation*}
for any $(n, g) \in \Z_{>0} \times G$. Let then $\G(\AYY)$ be the set of grouplike elements of $\AYY$ for the coproduct $\widehat{\Delta}_{G, \st}$, i.e. the set
\begin{equation*}
    \G(\AYY) = \{ \Psi \in \AYY^{\times} \, | \, \widehat{\Delta}_{G, \st}(\Psi) = \Psi \otimes \Psi \},
\end{equation*}
where $\AYY^{\times}$ denotes the set of invertible elements of $\AYY$. 

Recall the direct sum decomposition (of $\A$-submodules) (see (\cite[§2.2.5]{Rac02})
\[
    \AXX = \AYY \oplus \AXX x_0.
\]
Let then $\pi_{Y_G} : \AXX = \AYY \oplus \AXX x_0 \twoheadrightarrow \AYY$ be the projection from $\AXX$ to $\AYY$, that is, the surjective $\A$-module homomorphism such that it is the identity on $\AYY$ and maps any element of $\AXX x_0$ to $0$.

Define the graded $\A$-module automorphism $\qq_G$ of $\AX$ as the composition
\[
    \AX \simeq \A \otimes \h_G \xrightarrow{\mathrm{id}_\A \otimes \q_G} \A \otimes \h_G \to \AX.
\]
Its completion is an $\A$-module automorphism of $\AXX$ that we shall denote $\qq_G$ as well.
Finally, define the map
\[
    \AXX \to \A^{X_G^\ast}, \, \Phi \mapsto \big((\Phi | w)\big)_{w \in X^\ast_G},
\]
where $\Phi = \sum_{w \in X^\ast_G} (\Phi | w) w$.

\begin{definition}[{\cite[Definition 3.2.1]{Rac02}}]\label{def:dmrg}
    We define\footnote{The notation $\DMR$ is for "Double M\'{e}lange et R\'{e}gularisation" which is French for "Double Shuffle and Regularization".} $\DMR(G)(\A)$ to be the set\footnote{In \cite{Rac02}, this is the set denoted by $\underline{\DMR}(G)(\A)$.} of $\Phi \in \AXX$ with $(\Phi | x_0) = (\Phi | x_1) = 0$ such that
    \begin{enumerate}[label=(\alph*), leftmargin=*]
        \item $\widehat{\Delta}_{G, \sh}(\Phi) = \Phi \otimes \Phi$ and $(\Phi | 1) = 1$;
        \item $\widehat{\Delta}_{G, \st}(\Phi_{\st}) = \Phi_{\st} \otimes \Phi_{\st}$, where $\Phi_{\st} \df \Phi_{\mathrm{corr}} \,\, \pi_{Y_G}\left(\qq_G(\Phi)\right) \in \AYY$ and
        \[
            \Phi_{\mathrm{corr}} \df \exp\left( \sum_{n \geq 2} \frac{(-1)^{n-1}}{n} (\Phi | x_0^{n-1} x_1) x_1^n \right). 
        \]
    \end{enumerate}
\end{definition}

\begin{example}
    Set $\A=\C$, $G=\mu_N$ and $\zeta_N \df e^{\frac{\mathrm{i}2\pi}{N}}$. Define (see, for example, \cite{Fur13})
    \begin{align*}
        \Phi_{\mathrm{KZ}}^N := & 1 + \sum (-1)^r L_{(i_1, \dots, i_r)}(\zeta_N^{l_2-l_1}, \dots, \zeta_N^{l_r-l_{r-1}}, \zeta_N^{-l_r}) x_0^{i_r-1} x_{\zeta_N^{l_r}} \cdots x_0^{i_1-1} x_{\zeta_N^{l_1}} \\
        & + \text{(regularized terms)}  
    \end{align*}
    called the cyclotomic associator (see \cite{Enr08}), which for $N=1$ corresponds to the Drinfeld associator (see \cite{Dri91}). It satisfies
    \[
        \widehat{\Delta}_{G, \sh}(\Phi_{\mathrm{KZ}}^N) = \Phi_{\mathrm{KZ}}^N \otimes \Phi_{\mathrm{KZ}}^N \qquad
        \text{ and } \qquad \widehat{\Delta}_{G, \st}({\Phi_{\mathrm{KZ}}^N}_{\st}) = {\Phi_{\mathrm{KZ}}^N}_{\st} \otimes {\Phi_{\mathrm{KZ}}^N}_{\st},
    \]
    which at the level of coefficients correspond to the double shuffle and regularization relations of cyclotomic multiple zeta values (see \cite{Rac02}).
\end{example}

\subsection{Comparison between \texorpdfstring{$\EDS(G)(\A)$}{EDS(G)(A)} and \texorpdfstring{$\DMR(G)(\A)$}{DMR(G)(A)}}
Denote by $\ev_0^{\A} : \A[T] \to \A$ the evaluation map at $T=0$ and set
\[
    \rsh \df \ev_0^{\h^0_{\sh}} \circ \rsh^T,
\]
the unique $\Q$-algebra homomorphism $\h_{\sh} \to \h^0_{\sh}$ such that it is the identity on $\h^0$ and maps both $x_0$ and $x_1$ to $0$. \newline
Recall from Definition \ref{def:scalarproduct} the pairing
\begin{align*}
    \ALL \otimes \QL &\longrightarrow \A\\
    \psi \otimes w &\longmapsto  (\psi  \mid w),
\end{align*}
where $\mathcal{L} = X_G \text{ or } Y_G$. Direct computations enables one to prove that
\begin{equation}
    \label{qq_and_q}
    \qq_G(\Phi) = \sum_{w \in \mathcal{L}^\ast} (\Phi \mid \q_G^{-1}(w)) w.
\end{equation}

\begin{definition}\label{def:phiza}
    Given a $\Q$-linear map $Z_\A : \h^0_G \rightarrow \A$, we define the following element in $\AXX$
    \begin{align*}
        \Phi_{Z_\A \circ \rsh} \df \sum_{w \in X^\ast_G} Z_\A(\rsh(w)) w\,.
    \end{align*}
\end{definition}

\begin{theorem}
    \label{EDSiffDMR}
    $Z_\A \in \EDS(G)(\A) \Longleftrightarrow \Phi_{Z_\A \circ \rsh} \in \DMR(G)(\A)$.
\end{theorem}
\begin{proof}
    First, by definition of $\rsh$ and by linearity of $Z_\A$, we have for $i \in \{0, 1\}$,
    \[
        (\Phi_{Z_\A \circ \rsh} | x_i) = Z_\A(\rsh(x_i)) = 0.
    \]
    Next, since $\rsh : \h_{G, \sh} \rightarrow \h^0_{G, \sh}$ is an algebra homomorphism, the map $Z_\A : \h^0_{G, \sh} \rightarrow \A$ is an algebra homomorphism if and only if $Z_\A \circ \rsh : \h_{G, \sh} \rightarrow \A$ is an algebra homomorphism. Moreover, thanks to Proposition \ref{alghom_iff_grplk}, applied to $\LL=X_G$, $\varphi=Z_\A \circ \rsh$, $\qsh=\sh$ and $\widehat{\Delta}_{\qsh} = \widehat{\Delta}_{G, \sh}$, the map $Z_\A \circ \rsh : \h_{G, \sh} \rightarrow \A$ is an algebra homomorphism if and only if $\Phi_{Z_\A \circ \rsh} \in \G(\AXX)$. \newline
    Finally, since $\ev_0^\A$ is an algebra homomorphism, the map $\rho_{Z_\A}^{-1} \circ Z_\A^\sh \circ \q^{-1}_G : \h^1_{G, \st} \to \A[T]$ is an algebra homomorphism if and only if $\ev_0^\A \circ \rho_{Z_\A}^{-1} \circ Z_\A^\sh \circ \q^{-1}_G : \h^1_{G, \st} \to \A$ is an algebra homomorphism. Moreover, thanks to Proposition \ref{alghom_iff_grplk}, applied to $\LL=Y_G$, $\varphi = \ev_0^\A \circ \rho_{Z_\A}^{-1} \circ Z_\A^\sh \circ \q^{-1}_G$, $\qsh = \st$ and $\widehat{\Delta}_{\qsh} = \widehat{\Delta}_{G, \st}$, the map $\ev_0^\A \circ \rho_{Z_\A}^{-1} \circ Z_\A^{\sh} \circ \q^{-1}_G : \h^1_{G, \ast} \to \A[T]$ is an algebra homomorphism if and only if $\Phi_{\ev_0^\A \circ \rho_{Z_\A}^{-1} \circ Z^\sh_\A \circ \q^{-1}_G} \in \G(\AYY)$.
    Let us show that$\Phi_{\ev_0^\A \circ \rho_{Z_\A}^{-1} \circ Z^\sh_\A \circ \q^{-1}_G} \in \G(\AYY)$ if and only if $(\Phi_{Z_\A \circ \rsh})_{\st} \in \G(\AYY)$. In order to do so, it is enough to show that\footnote{For the special case of $Z_\C$  in Example \ref{ex:polylog} this was shown in \cite[Theorem 3.13]{YZ}.}     
    \begin{equation}
        \label{mainequality}
        (\Phi_{Z_\A \circ \rsh})_{\st} = \Phi_{\ev_0^\A \circ \rho_{Z_\A}^{-1} \circ Z^\sh_\A \circ \q^{-1}_G}. 
    \end{equation}
    \begin{steplist}
        \item Evaluation of $\Phi_{\ev_0^\A \circ \rho_{Z_\A}^{-1} \circ Z^\sh_\A \circ \q^{-1}_G}$. We have
        \begin{equation}
            \label{Phi_morph}
            \Phi_{\ev_0^\A \circ \rho_{Z_\A}^{-1} \circ Z^\sh_\A \circ \q^{-1}_G} = \sum_{w \in Y^\ast} \ev_0^\A \circ \rho_{Z_\A}^{-1} \circ Z^\sh_\A \circ \q^{-1}_G(w) \, w\,.
        \end{equation}
        \item Evaluation of $(\Phi_{Z_\A \circ \rsh})_{\st} = (\Phi_{Z_\A \circ \rsh})_{\mathrm{corr}} \,\, \pi_{Y_G}\left(\qq_G(\Phi_{Z_\A \circ \rsh})\right)$. In particular,
        \begin{align*}
            (\Phi_{Z_\A \circ \rsh})_{\mathrm{corr}} & = \exp\left(\sum_{n \geq 2}\frac{(-1)^{n-1}}{n}\left(\Phi_{Z_\A \circ \rsh} | x_0^{n-1} x_1\right) x_1^n \right) \\
            & = \exp\left(- \sum_{n \geq 2}\frac{(-1)^n}{n} Z_\A\left(x_0^{n-1} x_1\right) x_1^n\right) = \Gamma_{Z_\A}^{-1}(x_1),
        \end{align*}
        and
        \begin{align*}
            \pi_{Y_G} \circ \qq_G(\Phi_{Z_\A \circ \rsh}) & = \pi_{Y_G} \circ \qq_G\left(\sum_{w \in X^\ast_G} Z_\A(\rsh(w)) w \right) = \sum_{w \in X^\ast_G} Z_\A(\rsh(w)) \,\, \pi_{Y_G}(\qq_G(w)) \\
            & = \sum_{w \in Y^\ast_G} Z_\A(\reg_{\sh}(w)) \, \qq_G(w) = \sum_{w \in Y^\ast_G} Z_\A \circ \reg_{\sh} \circ \q^{-1}_G(w) \, w,
        \end{align*}
        where the second equality follows by linearity of $\pi_{Y_G} \circ \qq_G$, the third one by definitions of $\pi_{Y_G}$, $\qq_G$ and $\rsh$ and the fourth one from identity \eqref{qq_and_q}. 
        Therefore,
        \begin{equation}
            \label{Phi_star}
            (\Phi_{Z_\A \circ \rsh})_{\st} = \sum_{w \in Y^\ast_G} Z_\A \circ \reg_{\sh} \circ \q^{-1}_G(w) \, \Gamma_{Z_\A}^{-1}(x_1) \, w.
        \end{equation}
    \end{steplist}
    
    \noindent In order to prove equality \eqref{mainequality}, it suffices to prove that
    \[
        \left((\overline{\Phi}_{Z_\A})_{\st} | x_1^m w_0\right) = \left(\Phi_{\ev_{0} \circ \rho_{Z_\A}^{-1} \circ Z_\A^\sh \circ \q^{-1}_G} | x_1^m w_0\right), 
    \]
    for any $m \in \Z_{\geq 0}$ and any $w_0 \in \h^0_G$.
    \begin{steplist}
        \setcounter{steplisti}{2}
        \item \label{rho_inverse_eval}Evaluation of $\rho_{Z_\A}^{-1}$. Let $u$ be a formal parameter. \newline
        Let us write
        \begin{equation}
            \label{DecompGamma}
            \Gamma^{-1}_{Z_\A}(u) = \sum_{l \geq 0} \gamma_l u^l,
        \end{equation}
        where $\gamma_a$ are elements in $\A$ expressed in terms of images by $Z_\A$. This implies that
        \[
            \rho_{Z_\A}^{-1}\left(\frac{T^l}{l!}\right) = \sum_{j=0}^l \gamma_j \frac{T^{l-j}}{(l-j)!},
        \]
        for any $l \in \Z_{\geq 0}$.
        \item \label{reg_sha_eval}Evaluation of $\reg_{\sh}(x_1^n\q^{-1}_G(w_0))$ for any $n \in \Z_{\geq 0}$ and $w_0 \in \h^0_G$. \newline
        Recall that $\q^{-1}_G(w_0) \in \h^0_G$, since $w_0 \in \h^0_G$. Therefore, one may write $\q^{-1}_G(w_0) = x_g w_0^\prime$ for some $g \in (G \sqcup \{0\}) \smallsetminus \{1\}$ and $w_0^\prime \in \h^1_G$. By Corollary \ref{cor:regformula} for the case $w=\q^{-1}_G(w_0)$ we obtain
        \[
            \reg_{\sh}^T\left(\sum_{n \geq 0} x_1^n u^n \q^{-1}_G(w_0)\right) = x_g \left(\sum_{k \geq 0} (-1)^k x_1^k u^k \sh w_0^\prime\right) \sum_{l\geq 0} \frac{T^l}{l!} u^l,
        \]
        for a formal parameter $u$. This implies that
        \begin{equation}
            \label{regshTExpression}
            \reg_{\sh}^T(x_1^n \q^{-1}_G(w_0)) = \sum_{\substack{k,l\geq 0 \\ k+l=n}} \frac{(-1)^k}{l!} x_g (x_1^k \sh w_0^\prime) T^l 
        \end{equation}
        for any $n \in \Z_{\geq 0}$, and therefore
        \begin{equation}
            \label{regshExpression}
            \reg_{\sh}(x_1^n\q^{-1}_G(w_0)) = (-1)^n x_g (x_1^n \sh w_0^\prime). 
        \end{equation}
        \item Evaluation of $\left((\Phi_{Z_\A \circ \rsh})_{\st} | x_1^m w_0\right)$. \newline
        Injecting in \eqref{Phi_star} the expressions \eqref{DecompGamma} and \eqref{regshExpression} from \ref{rho_inverse_eval} and \ref{reg_sha_eval} respectively, we obtain
        \begin{align*}
            \left((\Phi_{Z_\A \circ \rsh})_{\st} | x_1^m w_0\right) & = \sum_{\substack{k,l \geq 0 \\ k+l=m}} Z_\A \circ \reg_{\sh} \circ \q^{-1}_G(x_1^k w_0) \gamma_l = \sum_{\substack{k,l \geq 0 \\ k+l=m}} Z_\A \circ \reg_{\sh}(x_1^k\q^{-1}_G(w_0)) \gamma_l \\
            & = \sum_{\substack{k,l \geq 0 \\ k+l=m}} (-1)^k Z_\A \left(x_g (x_1^k \sh w_0^\prime)\right) \gamma_l.
        \end{align*}
        \item Evaluation of $\left(\Phi_{\ev_0^\A \circ \rho_{Z_\A}^{-1} \circ Z^\sh_\A \circ \q^{-1}_G} | x_1^m w_0\right)$.
        We have
        \begin{align*}
            \left(\Phi_{\ev_0^\A \circ \rho_{Z_\A}^{-1} \circ Z^\sh_\A \circ \q^{-1}_G} | x_1^m w_0\right) & = \ev_0^\A \circ \rho_{Z_\A}^{-1} \circ Z^\sh_\A \circ \q^{-1}_G(x_1^m w_0) = \ev_0^\A \circ \rho_{Z_\A}^{-1} \circ Z^\sh_\A (x_1^m \q^{-1}_G(w_0)) \\
            & = \ev_0^\A \circ \rho_{Z_\A}^{-1} \left( \sum_{\substack{k,l\geq 0 \\ k+l=m}} (-1)^k Z_\A\left(x_g (x_1^k \sh w_0^\prime)\right) \frac{T^l}{l!} \right) \\
            & = \ev_0^\A \left( \sum_{\substack{k,l\geq 0 \\ k+l=m}} (-1)^k Z_\A\left(x_g (x_1^k \sh w_0^\prime)\right) \sum_{j=0}^l \gamma_j \frac{T^{l-j}}{(l-j)!} \right) \\
            & = \sum_{\substack{k,l\geq 0 \\ k+l=m}} (-1)^k Z_\A\left(x_g (x_1^k \sh w_0^\prime)\right) \gamma_l,
        \end{align*}
        where the first equality comes from expression \eqref{Phi_morph} and the third one from the Composition \eqref{compositionZsh} which interprets $Z_\A^\sh$ as a composition involving $\reg_{\sh}^T$ which enables us to use expression \eqref{regshTExpression}.
    \end{steplist}
    Therefore, we have
    \[
        \left((\Phi_{Z_\A \circ \rsh})_{\st} | x_1^m w_0\right) = \sum_{\substack{k,l \geq 0 \\ k+l=m}} (-1)^k Z_\A \left(x_g (x_1^k \sh w_0^\prime)\right) \gamma_l = \left(\Phi_{\ev_0^\A \circ \rho_{Z_\A}^{-1} \circ Z^\sh_\A \circ \q^{-1}_G} | x_1^m w_0\right),
    \]
    which establishes equality \eqref{mainequality}.  
\end{proof}

\begin{theorem}
    \label{main_bijection}
    The map
    \begin{equation}\label{eq:bijmap}
        \begin{array}{ccc}
             \EDS(G)(\A) & \longrightarrow & \DMR(G)(\A) \\
             Z_\A & \longmapsto & \Phi_{Z_\A \circ \rsh}
        \end{array}
    \end{equation}
    is bijective.
\end{theorem}

\noindent In order to prove this, we will need the following result:
\begin{lemma}\label{h0h_extension}
    Let $Z_{\A} :\h^0_{G, \sh} \to \A$ be a $\Q$-algebra homomorphism. Then $Z_{\A} \circ \rsh :\h_{G, \sh} \to \A$ is the unique $\Q$-algebra homomorphism such that it agrees with $Z_{\A}$ on $\h^0_G$ and maps $x_0$ to $0$ and $x_1$ to $0$.    
\end{lemma}
\begin{proof}
    This follows by definition of $\rsh$.
\end{proof}

\begin{proof}[Proof of Theorem \ref{main_bijection}]
    Let us start by showing that the map \eqref{eq:bijmap} is injective. Indeed, let $Z_\A$ and $Z_\A^{\prime} \in \EDS(G)(\A)$ such that $\Phi_{Z_\A \circ \rsh} = \Phi_{Z^\prime_\A \circ \rsh}$. This implies that $Z_\A \circ \rsh(w)=Z^\prime_\A \circ \rsh(w)$ for any $w \in X^\ast$. In particular, for any $w_0 \in X^\ast \cap \h^0_G$, one obtains $Z_\A(w)=Z^\prime_\A(w)$. Then one concludes the equality by linearity. \newline
    Now, let us show surjectivity. Let $\Phi \in \DMR(G)(\A)$. Define the linear map $Z_\A : \h^0_G \to \A$ such that $w_0 \mapsto (\Phi | w_0)$, for $w_0 \in X^\ast \cap \h^0_G$. Thanks to Theorem \ref{EDSiffDMR}, the result follows if we show that $\Phi = \Phi_{Z_\A \circ \rsh}$. This is equivalent to show that
    \begin{equation}
        \label{extension_identity}
        (\Phi | w) = Z_\A \circ \rsh(w),
    \end{equation}
    for any $w \in X^\ast_G$.
    Let us consider the linear map $z_\A : \h_G \to \A$ such that $w \mapsto (\Phi | w)$ for any $w \in X^\ast_G$. It is immediate that this map agrees with $Z_\A$ on $\h^0_G$ and maps $x_0$ and $x_1$ to $0$. In addition, since $\Phi \in \G(\AXX)$, then $z_\A$ is an algebra homomorphism $\h_{G, \sh} \to \A$.
    On the other hand, the grouplikeness of $\Phi$ also implies that the map $Z_\A : \h^0_{G, \sh} \to \A$ is an algebra homomorphism. Therefore, thanks to Lemma \ref{h0h_extension}, $Z_\A \circ \rsh : \h_{G, \sh} \to \A$ is the unique algebra morphism such that it is $Z_\A$ on $\h^0_G$ and maps $x_0$ and $x_1$ to $0$. This proves that $z_\A = Z_\A \circ \rsh$ and thus identity \eqref{extension_identity}.
\end{proof}

    \section{Distribution relations}
Throughout this section, $G$ is a non-trivial group, $d \in \Z_{>0}$ is a divisor of the order of $G$ and consider the subgroup
\[
    G^d \df \{ g^d \, | \, g \in G\}.
\]
of $G$.
This section extends algebraic relations between multiple polylogarithm values at roots of unity to include distribution relations. We begin by examining Racinet's approach, which encodes these additional relations into a subscheme $\mathsf{DMRD}(G)$ of $\mathsf{DMR}(G)$. Parallel to this, we introduce a subscheme $\mathsf{EDSD}(G)$ of $\mathsf{EDS}(G)$, which captures the distribution relations within the framework developed by Arakawa and Kaneko \cite{AK}. Thanks to the correspondence of schemes established in the previous section, we prove that these subschemes are indeed isomorphic. Finally, we apply this equivalent formalism to prove the conjecture of Zhao \cite[Conjeture 4.7]{Zha10} stated in the introduction.

\subsection{The homomorphisms \texorpdfstring{$i_d^\ast$}{idstar}, \texorpdfstring{$p^d_\ast$}{pdstar}, \texorpdfstring{$i_d^\sharp$}{idsharp} and \texorpdfstring{$p_d^\sharp$}{pdsharp}}
Let $p^d : G \twoheadrightarrow G^d$ be the group homomorphism $g \mapsto g^d$ and $i_d : G^d \hookrightarrow G$ be the canonical inclusion. \newline
Applying the functor given by Proposition-Definition \ref{star_functors} \ref{upper_star} for $\phi = i_d$, we define the Hopf algebra homomorphism $i_d^\ast : (\AXX, \widehat{\Delta}_{G,\sh}) \to (\AXXGd, \widehat{\Delta}_{G^d,\sh})$. One checks that this homomorphism is explicitly given by
\[
    i_d^\ast : \quad \begin{aligned}
        &x_0 \mapsto x_0 \\
        &x_g \mapsto \begin{cases} x_g & \text{if } g \in G^d \\ 0 & \text{otherwise}\end{cases}
    \end{aligned}
\]
Applying the functor given by Proposition-Definition \ref{star_functors} \ref{lower_star} for $\phi = p^d$, we define the Hopf algebra homomorphism $p^d_\ast : (\AXX, \widehat{\Delta}_{G,\sh}) \to (\AXXGd, \widehat{\Delta}_{G^d,\sh})$. One checks that this homomorphism is explicitly given by (see also \cite[§2.5.3]{Rac02} and \cite[§4]{Zha10})
\[
    p^d_\ast : \quad \begin{aligned}
        &x_0 \mapsto d \, x_0 \\
        &x_g \mapsto x_{g^d}
    \end{aligned}
\]
Applying the functor given by Proposition-Definition \ref{sharp_functors} \ref{upper_sharp} for $\phi = i_d$, we define the algebra homomorphism $i_d^\sharp : \h_{G^d} \to \h_G$. One checks that this homomorphism is explicitly given by (see also \cite[§2.5.3]{Rac02} and \cite[§4]{Zha10})
\[
    i_d^\sharp : \quad \begin{aligned}
        &x_0 \mapsto x_0 \\
        &x_h \mapsto x_h 
    \end{aligned}
\]
Applying the functor given by Proposition-Definition \ref{sharp_functors} \ref{lower_sharp} for $\phi = p^d$, we define the algebra homomorphism $p^d_\sharp : \h_{G^d} \to \h_G$. One checks that this homomorphism is explicitly given by
\[
    p^d_\sharp : \quad \begin{aligned}
        &x_0 \mapsto d \, x_0 \\
        &x_h \mapsto \sum_{g^d=h} x_g
    \end{aligned}
\]

Recall from Definition \ref{def:scalarproduct} the pairing
\((-, -)_G : \A\langle\langle X_G\rangle\rangle \otimes \h_G \to \A\) (resp. \((-, -)_{G^d} : \A\langle\langle X_{G^d}\rangle\rangle \otimes \h_{G^d} \to \A\)) for $\mathcal{L}=X_G$ (resp. $\mathcal{L}=X_{G^d}$).
\begin{lemma}
    We have
    \begin{enumerate}[label=(\alph*), leftmargin=*]
        \item $(i_d^\ast(S_2), P_1)_{G^d} = (S_2, i_d^\sharp(P_1))_G$, for $S_2 \in \A\langle\langle X_{G}\rangle\rangle$ and $P_1 \in \h_{G^d}$.
        \item $(p^d_\ast(S_1), P_2)_{G^d} = (S_1, p^d_\sharp(P_2))_G$, for $S_1 \in \A\langle\langle X_{G}\rangle\rangle$ and $P_2 \in \h_{G^d}$.
    \end{enumerate}
\end{lemma}
\begin{proof}
    This follows immediately by applying Lemma \ref{lem:astsharp} to $\phi = i_d$, $G_1 = G^d$ and $G_2 = G$ (resp. $\phi = p^d$, $G_1 = G$ and $G_2 = G^d$).
\end{proof}

\subsection{The schemes \texorpdfstring{$\mathsf{DMRD}$}{DMRD} and \texorpdfstring{$\mathsf{EDSD}$}{EDSD} of distribution relations}
\begin{definition}[{\cite[Definition 3.2.1]{Rac02}}]
    We define $\DMRD(G)(\A)$ to be the set\footnote{In \cite{Rac02}, this is the set denoted by $\underline{\DMRD}(G)(\A)$.} of $\Phi \in \DMR(G)(\A)$ such that for every divisor $d$ of the order of $G$, we have
    \[
        p_\ast^d (\Phi) = \exp\Bigg(\sum_{g^d =1} (\Phi | x_g) x_1 \Bigg) i_d^\ast (\Phi).
    \]
\end{definition}

\begin{definition}
    A $\Q$-linear map $Z_\A : \h^0_G \to \A$ satisfies \emph{finite distribution} relations if for every divisor $d$ of the order of $G$, we have (equality of $\Q$-linear maps $\h^0_{G^d} \to \A$)
    \[
        Z_\A \circ p_\sharp^d = Z_\A \circ i^\sharp_d.
    \]
\end{definition}

\begin{example}
    Let $N > 1$ be a positive integer. Finite distribution relations of multiple polylogarithm values at $N$\textsuperscript{th} roots of unity are expressed by (see \cite{Gon01})
    \begin{equation}
        \label{dist_polylog}
        \Li_{(k_1, \dots, k_r)}(z_1, \dots, z_r) = d^{k_1 + \cdots + k_r - r} \sum_{\substack{t_i^d=z_i \\ 1 \leq i \leq r}} \Li_{(k_1, \dots, k_r)}(t_1, \dots, t_r),
    \end{equation}
    for any divisor $d$ of $N$. Recall from Example \ref{ex:polylog} the $\Q$-linear map $Z_\C : \h^0_{\mu_N} \to \C$. For any divisor $d$ of $N$ and any word $x_0^{k_1-1}x_{z_1} \cdots x_0^{k_r-1}x_{z_r}$ of $\h^0_{\mu_N^d}$, we have
    \[
        Z_\C \circ \q_{\mu_N}^{-1} \circ p_\sharp^d(x_0^{k_1-1}x_{z_1} \cdots x_0^{k_r-1}x_{z_r}) = d^{k_1 + \cdots + k_r - r} \sum_{\substack{t_i^d=z_i \\ 1 \leq i \leq r}} Z_\C \circ \q_{\mu_N}^{-1}(x_0^{k_1-1}x_{t_1} \cdots x_0^{k_r-1}x_{t_r}), 
    \]
    which is equal to the right hand side of \eqref{dist_polylog} thanks to \eqref{polylog_Zcircq}. On the other hand,
    \[
        Z_\C \circ \q_{\mu_N}^{-1} \circ i^\sharp_d(x_0^{k_1-1}x_{z_1} \cdots x_0^{k_r-1}x_{z_r}) = Z_\C \circ \q_{\mu_N}^{-1}(x_0^{k_1-1}x_{z_1} \cdots x_0^{k_r-1}x_{z_r}), 
    \]
    which is the left hand side of \eqref{dist_polylog} thanks to \eqref{polylog_Zcircq}. Therefore, equality \eqref{dist_polylog} is satisfied if and only if
    \[
        Z_\C \circ \q_{\mu_N}^{-1} \circ p_\sharp^d = Z_\C \circ \q_{\mu_N}^{-1} \circ i^\sharp_d.
    \]
    Thanks to Corollary \ref{sharp_comm} \eqref{comm_q_up} (resp. \eqref{comm_q_down}) applied to $\phi = i_d$, $G_1 = \mu_N$ and $G_2 = \mu_N^d$ (resp. $\phi = p^d$, $G_1 = \mu_N^d$ and $G_2 = \mu_N$), this is equivalent to
    \[
        Z_\C \circ p_\sharp^d \circ \q_{\mu_N^d}^{-1} = Z_\C \circ i^\sharp_d \circ \q_{\mu_N^d}^{-1}.
    \]
    Hence, equality \eqref{dist_polylog} is satisfied if and only if
    \[
        Z_\C \circ p_\sharp^d = Z_\C \circ i^\sharp_d.
    \]
\end{example}

\begin{definition}
    For a $\Q$-linear map $Z_\A : \h^0_G \to \A$, we define the $\A$-module automorphism $\sigma_{Z_\A}$ of $\A[T]$ such that
    \[
        \sigma_{Z_\A}(\exp(Tu)) = \exp\Bigg(\sum_{\substack{g^d=1\\ g\neq1}} Z_{\A}(x_g) u \Bigg) \exp(Tu),
    \]
for a formal parameter $u$. 
\end{definition}

Let us write
\begin{equation}
    \label{exp_dec}
    \exp\Bigg(\sum_{\substack{g^d=1 \\ g\neq1}} Z_\A(x_g) u\Bigg) = \sum_{l\geq 0} \delta_l u^l,    
\end{equation}
where $\delta_0 = 1$, $\delta_1 = \sum_{\substack{g^d=1\\ g\neq1}} Z_{\A}(x_g)$ and $\delta_l = \frac{\delta_1^l}{l!}$ for $l \geq 0$. It follows that
\begin{equation}
    \label{sigma_dec}
    \sigma_{Z_\A}\left(\frac{T^l}{l!}\right) = \sum_{j=0}^l\delta_j \frac{T^{l-j}}{(l-j)!}. 
\end{equation}

Additionally, for a word $w = x_h w^\prime$ with $h \in (G^d \sqcup \{0\}) \smallsetminus \{1\}$ and $w^\prime \in X^\ast_{G^d}$, we obtain from Lemma \ref{lem:barregformula} that
\[
    \overline{\reg}_\sh^T\left(\sum_{n \geq 0} x_1^n u^n w \right) = x_h \left(\sum_{k \geq 0} (-1)^k x_1^k u^k \, \sh \, \widetilde{\reg}_\sh(w^\prime)\right) \sum_{l\geq 0} \frac{T^l}{l!} u^l.
\]
In particular, this implies that
\begin{equation}
    \label{barregshTExpression}
    \overline{\reg}_{\sh}^T(x_1^m w) = \sum_{\substack{k,l\geq 0 \\ k+l=m}} \frac{(-1)^k}{l!} x_h (x_1^k \, \sh \, \widetilde{\reg}_\sh(w^\prime)) T^l, 
\end{equation}
and therefore
\begin{equation}
    \label{barregshExpression}
    \overline{\reg}_{\sh}(x_1^m w) = (-1)^m x_h (x_1^m \, \sh \, \widetilde{\reg}_\sh(w^\prime)). 
\end{equation}

\begin{definition}
    A $\Q$-algebra homomorphism $Z_\A : \h_{G, \sh}^0 \to \A$ satisfies \emph{regularized distribution} relations if for every divisor $d$ of the order of $G$, we have the following equality of $\Q$-linear maps $\h_{G^d} \to \A[T]$
    \[
        \overline{Z}_\A^\sh \circ p_\sharp^d = \sigma_{Z_\A} \circ \overline{Z}_\A^\sh \circ i^\sharp_d.
    \]  
\end{definition}

\begin{proposition}
    \label{prop:dist_iff_noeval}
    An algebra homomorphism $Z_\A : \h_\sh^0 \to \A$ satisfies regularized distribution relations if and only if for every divisor $d$ of the order of $G$, we have an equality of $\Q$-linear maps $\h_{G^d} \to \A$
    \[
        \ev_0^\A \circ \overline{Z}_\A^\sh \circ p_\sharp^d = \ev_0^\A \circ \sigma_{Z_\A} \circ \overline{Z}_\A^\sh \circ i^\sharp_d.
    \]    
\end{proposition}

\begin{proof}
    ($\Rightarrow$) is immediate. Let us prove ($\Leftarrow$). In the following, we fix $d$ a divisor of the order of $G$.
    \begin{steplist}
        \item \label{Zpx1m_Zix1m}For any $m \in \Z_{\geq 0}$, let us show that
        \[
            \overline{Z}_\A^\sh \circ p_\sharp^d(x_1^m) = \sigma_{Z_\A} \circ \overline{Z}_\A^\sh \circ i^\sharp_d(x_1^m).
        \]
        We have
        \begin{align*}
            \sigma_{Z_\A} \circ \overline{Z}_\A^\sh \circ i^\sharp_d(x_1^m) & = \sigma_{Z_\A} \circ \overline{Z}_\A^\sh \circ i^\sharp_d\left(\frac{x_1^{\sh m}}{m!}\right) = \sigma_{Z_\A} \circ \overline{Z}_\A^\sh\left(\frac{x_1^{\sh m}}{m!}\right) \\
            & = \sigma_{Z_\A}\left(\frac{\overline{Z}_\A^\sh(x_1)^m}{m!}\right) = \sigma_{Z_\A}\left(\frac{T^m}{m!}\right) = \sum_{j=0}^m \delta_j \frac{T^{m-j}}{(m-j)!}, 
        \end{align*}
        where the first equality comes from $x_1^{\sh m} = x_1 \, \sh x_1 \, \sh \, \cdots \, \sh \, x_1 = m! \, x_1^m$, the third one from the fact that $\overline{Z}_\A^\sh : \h_{G, \sh} \to \A[T]$ is an algebra homomorphism and the last one from equality \eqref{sigma_dec}. On the other hand, we have
        \begin{align*}
            \overline{Z}_\A^\sh \circ p_\sharp^d(x_1^m) & = \overline{Z}_\A^\sh \circ p_\sharp^d\left(\frac{x_1^{\sh m}}{m!}\right) = \overline{Z}_\A^\sh\left(\frac{p_\sharp^d(x_1)^{\sh m}}{m!}\right) = \frac{1}{m!} \left(\overline{Z}_\A^\sh \circ p_\sharp^d(x_1)\right)^m \\
            & = \frac{1}{m!} \Bigg(\overline{Z}_\A^\sh \Bigg(\sum_{g^d = 1} x_g\Bigg)\Bigg)^m = \frac{1}{m!} \Bigg(\overline{Z}_\A^\sh \Bigg(x_1 + \sum_{\substack{g^d = 1\\g \neq 1}} x_g\Bigg)\Bigg)^m \\
            & = \frac{1}{m!} \Bigg(T + Z_\A\Bigg(\sum_{\substack{g^d = 1\\g \neq 1}} x_g\Bigg)\Bigg)^m = \frac{1}{m!}(\delta_0 T + \delta_1)^m \\
            & = \frac{1}{m!} \sum_{j=0}^m \frac{m!}{j!(m-j)!} \delta_1^j T^{m-j} = \sum_{j=0}^m \delta_j \frac{T^{m-j}}{(m-j)!},
        \end{align*}
        where the second equality comes from the fact that $p_\sharp^d : \h_{G^d, \sh} \to \h_{G, \sh}$ is an algebra homomorphism thanks to Corollary \ref{cor:alg_morph} \ref{cor:sh_alg_morph} and the last equality comes from the identity $\delta_j = \frac{\delta_1^j}{j!}$ for $j \geq 1$.
        \item \label{Zregp_Zregi}For a word $w$ of $\h_{G^d, \sh}$ that does not start with $x_1$, let us show that
        \[
            Z_\A \circ \rsh \circ p_\sharp^d(w) = Z_\A \circ \rsh \circ i^\sharp_d(w). 
        \]
        By assumption on $w$, let us write $w = x_h w^\prime$ with $h \in (G^d \cup \{0\}) \smallsetminus \{1\}$ and $w^\prime \in X^\ast_{G^d}$. We have
        \begin{align*}
            Z_\A \circ \rsh \circ i^\sharp_d(w) & = Z_\A \circ \rsh \circ i^\sharp_d(x_h w^\prime) = Z_\A \circ \rsh (x_h w^\prime) \\
            & = Z_\A\left(x_h (1 \, \sh \, \widetilde{\reg}_\sh(w^\prime) \right) = Z_\A\left( x_h \widetilde{\reg}_\sh(w^\prime) \right),
        \end{align*}
        where the third equality comes from \eqref{barregshExpression}. On the other hand,
        \begin{align*}
            Z_\A \circ \rsh \circ p_\sharp^d(w) & = \ev_0^\A \circ \overline{Z}_\A^\sh \circ p_\sharp^d(w) = \ev_0^\A \circ \sigma_{Z_\A} \circ \overline{Z}_\A^\sh \circ i^\sharp_d(w) \\
            & = \ev_0^\A \circ \sigma_{Z_\A} \circ \overline{Z}_\A^\sh \circ i^\sharp_d(x_h w^\prime) = \ev_0^\A \circ \sigma_{Z_\A} \circ \overline{Z}_\A^\sh(x_h w^\prime) \\
            & = \ev_0^\A \circ \sigma_{Z_\A} \left(Z_\A(x_h (1 \, \sh \, \widetilde{\reg}_\sh(w^\prime))\right) = \ev_0^\A\left(Z_\A(x_h \, \widetilde{\reg}_\sh(w^\prime))\right) \\
            & = Z_\A\left( x_h \widetilde{\reg}_\sh(w^\prime) \right),
        \end{align*}
        where the first equality comes from Corollary \ref{Zreg}, the second one from assumption on $Z_\A$ and the fifth one from expression \eqref{barregshExpression}.
        \item \label{barZsha_Zcircbarregsha}For a word $v$ of $\h_{G, \sh}$ that does not start with $x_1$, let us show that
        \[
            \overline{Z}_\A^\sh(v) = Z_A \circ \rsh(v).
        \]
        By assumption on $v$, let us write $v = x_g v^\prime$ with $g \in (G \cup \{0\}) \smallsetminus \{1\}$ and $v^\prime \in X^\ast_G$. We have
        \[
            \overline{Z}_\A^\sh(v) = \overline{Z}_\A^\sh(x_h v^\prime) = Z_\A(x_h \widetilde{\reg}_\sh(v^\prime)) = Z_A \circ \rsh(x_h v^\prime) = Z_A \circ \rsh(v),
        \]
        where the second equality comes from the composition \eqref{compositionZsh} which interprets $\overline{Z}_\A^\sh$ as a composition involving $\overline{\reg}_{\sh}^T$ and enables the use of expression \eqref{barregshTExpression} and the third equality comes from expression \eqref{barregshExpression}.
        \item \label{Zpw_Ziw}For a word $w$ of $\h_{G^d, \sh}$ that does not start with $x_1$, let us show that
        \[
            \overline{Z}_\A^\sh \circ p_\sharp^d(w) = \overline{Z}_\A^\sh \circ i^\sharp_d(w). 
        \]
        Since $w \in X_{G^d}^\ast$ does not start with $x_1$, so is $i^\sharp_d(w)$. Therefore, from \ref{barZsha_Zcircbarregsha}, it follows that
        \begin{equation}
            \label{barZisharp}
            \overline{Z}_\A^\sh \circ i^\sharp_d(w) = Z_A \circ \rsh \circ i^\sharp_d(w).
        \end{equation}
        On the other hand, since $w \in X_{G^d}^\ast$ does not start with $x_1$, $p_\sharp^d(w)$ is a linear combination of words in $X_G^\ast$ that do not start with $x_1$. Therefore, from \ref{barZsha_Zcircbarregsha}, it follows that
        \begin{equation}
            \label{barZpsharp}
            \overline{Z}_\A^\sh \circ p_\sharp^d(w) = Z_A \circ \rsh \circ p_\sharp^d(w).
        \end{equation}
        Then
        \[
            \overline{Z}_\A^\sh \circ p_\sharp^d(w) = Z_A \circ \rsh \circ p_\sharp^d(w) = Z_A \circ \rsh \circ i^\sharp_d(w) = \overline{Z}_\A^\sh \circ i^\sharp_d(w),
        \]
        where the first equality follows from \eqref{barZpsharp}; the second one from \ref{Zregp_Zregi} and the last one from \eqref{barZisharp}.
        \item \label{Zpx1mw_Zix1mw}For $m \in \Z_{\geq 0}$ and a word $w$ of $\h_{G^d, \sh}$ that does not start with $x_1$, let us show that
        \[
            \overline{Z}_\A^\sh \circ p_\sharp^d(x_1^m \, \sh \, w) = \sigma_{Z_\A} \circ \overline{Z}_\A^\sh \circ i^\sharp_d(x_1^m \, \sh \, w). 
        \]
        We have
        \begin{align*}
            \overline{Z}_\A^\sh \circ p_\sharp^d(x_1^m \, \sh \, w) & = \overline{Z}_\A^\sh \left(p_\sharp^d(x_1^m) \, \sh \, p_\sharp^d(w)\right) = \overline{Z}_\A^\sh \circ p_\sharp^d(x_1^m) \, \, \overline{Z}_\A^\sh \circ p_\sharp^d(w) \\
            & = \sigma_{Z_\A} \circ \overline{Z}_\A^\sh \circ i^\sharp_d(x_1^m) \, \, \overline{Z}_\A^\sh \circ i^\sharp_d(w) = \sigma_{Z_\A} \left(\overline{Z}_\A^\sh \circ i^\sharp_d(x_1^m) \, \, \overline{Z}_\A^\sh \circ i^\sharp_d(w) \right) \\
            & = \sigma_{Z_\A} \left(\overline{Z}_\A^\sh \left(i^\sharp_d(x_1^m) \, \sh \, i^\sharp_d(w) \right)\right) = \sigma_{Z_\A} \circ \overline{Z}_\A^\sh \circ i^\sharp_d(x_1^m \, \sh \, w), 
        \end{align*}
        where the first (resp. last) equality comes from the fact that $p_\sharp^d : \h_{G^d, \sh} \to \h_{G, \sh}$ (resp. $i^\sharp_d : \h_{G^d, \sh} \to \h_{G, \sh}$) is an algebra homomorphism thanks to Corollary \ref{cor:alg_morph} \ref{cor:sh_alg_morph}, the second and fifth ones from the fact that $\overline{Z}_\A^\sh : \h_\sh \to \A[T]$ is an algebra homomorphism and the third one from \ref{Zpx1m_Zix1m} and \ref{Zpw_Ziw}, the fourth one from the $\A$-linearity of $\sigma_{Z_\A}$ and the fact that $\overline{Z}_\A^\sh \circ i^\sharp_d(w) \in \A$ thanks to \ref{barZsha_Zcircbarregsha}.
    \end{steplist}
    Finally, since all words in $\h_{G^d}$ can be generated by the elements $x_1^m \, \sh \, w$, with $m \in \Z_{\geq 0}$ and $w \in X_{G^d}^\ast$ that does not start with $x_1$, the equality of \ref{Zpx1mw_Zix1mw} suffices to prove that
     \[
        Z_\A \circ p_\sharp^d = Z_\A \circ i^\sharp_d,
    \]
    which is the wanted result.
\end{proof}

\begin{definition}\label{def:edsd}
    We define $\EDSD(G)(\A)$ to be the set of elements $Z_\A \in \EDS(G)(\A)$ such that the algebra homomorphism $Z_\A : \h_{G, \sh}^0 \to \A$ satisfies regularized distribution relations.  
\end{definition}

\begin{theorem}
    \label{main_bijection_2}
    The bijection $\EDS(G)(\A) \to \DMR(G)(\A)$, $Z_\A \mapsto \Phi_{Z_\A \circ \rsh}$ restricts to a bijection
    \[
        \EDSD(G)(\A) \longrightarrow \DMRD(G)(\A).
    \]
\end{theorem}
\begin{proof}
    Thanks to Theorem \ref{main_bijection}, this is equivalent to show that for $ Z_\A \in \EDS(G)(\A)$
    \[
        Z_\A \in \EDSD(G)(\A) \Longleftrightarrow \Phi_{Z_\A \circ \rsh} \in \DMRD(G)(\A).
    \]
    In the following, we fix a divisor $d$ of the order of $G$.
    We have
    \begin{align*}
        p_\ast^d (\Phi_{Z_\A \circ \rsh}) & = \sum_{w \in X^\ast_G} Z_\A \circ \rsh(w) p^d_\ast(w) = \sum_{w \in X^\ast_{G^d}} Z_\A \circ \rsh \circ p^d_\sharp(w) \, w \\
        & = \sum_{w \in X^\ast_{G^d}} \ev_0^\A \circ \overline{Z}_\A^\sh \circ p^d_\sharp(w) \, w,
    \end{align*}
    where the second equality follows from Lemma \ref{lem:astsharp} applied to $\phi=p^d$, and the third one from Corollary \ref{Zreg}.
    On the other hand, we have
    \begin{align*}
        & \exp\Bigg(\sum_{g^d=1} (\Phi_{Z_\A \circ \rsh}|x_g) x_1 \Bigg) i^\ast_d(\Phi_{Z_\A \circ \rsh}) = \exp\Bigg(\sum_{\substack{g^d=1 \\ g\neq 1}} Z_\A(x_g) x_1 \Bigg) \sum_{w \in X^\ast_G} Z_\A \circ \rsh(w) i^\ast_d(w) \\
        & = \sum_{w \in X^\ast_{G^d}} Z_\A \circ \rsh \circ i_d^\sharp(w) \exp\Bigg(\sum_{\substack{g^d=1 \\ g\neq 1}} Z_\A(x_g) x_1 \Bigg) \, w,
    \end{align*}
    where the second equality comes from Lemma \ref{lem:astsharp} applied to $\phi = i_d$. \newline
    Let $m \geq 0$ and $w \in X^\ast_{G^d}$ that does not starting with $x_1$. We have
    \begin{equation}
        \label{dist:lhs}
        \left( p_\ast^d (\Phi_{Z_\A \circ \rsh}) \Big| x_1^m w \right) = \ev_0^\A \circ \overline{Z}_\A^\sh \circ p^d_\sharp (x_1^m w).
    \end{equation}
    On the other hand, we have
    \begin{align}
        & \Bigg(\exp\Bigg(\sum_{g^d=1} (\Phi_{Z_\A \circ \rsh}|x_g) x_1 \Bigg) i^\ast_d(\Phi_{Z_\A \circ \rsh}) \Big| x_1^m w \Bigg) = \sum_{\substack{k+l=m \\ k,l\geq0}} Z_\A \circ \rsh \circ i_d^\sharp(x_1^k w) \delta_l \notag\\
        & = \sum_{\substack{k+l=m \\ k,l\geq0}} Z_\A\left((-1)^k x_h\left(x_1^k \, \sh \, \widetilde{\reg}_\sh(w^\prime)\right)\right) \delta_l, \label{dist:rhs}
    \end{align}
    where the first equality comes from identity \eqref{dist:rhs} and expression \eqref{exp_dec} and the second one from expression \eqref{barregshExpression}.
    \begin{caselist}
        \item Assume that $Z_\A \in \EDSD(G)(\A)$. Then, for any divisor $d$ of the order of $G$, $m \geq 0$ and $w \in X^\ast_{G^d}$ that does not starting with $x_1$, we have
        \begin{align*}
            \left( p_\ast^d (\Phi_{Z_\A \circ \rsh}) \Big| x_1^m w \right) & = \ev_0^\A \circ \sigma_{Z_\A} \circ \overline{Z}_\A^\sh \circ i_d^\sharp(x_1^m w) = \ev_0^\A \circ \sigma_{Z_\A} \circ \overline{Z}_\A^\sh (x_1^m w) \\
            & = \ev_0^\A \circ \sigma_{Z_\A}\Bigg( \sum_{\substack{k+l=m \\ k,l\geq0}} \frac{(-1)^k}{l!} Z_\A\left(x_h \left(x_1^k \, \sh \, \widetilde{\reg}_\sh(w^\prime)\right)\right) T^l\Bigg) \\
            & = \ev_0^\A \Bigg( \sum_{\substack{k+l=m \\ k,l\geq0}} (-1)^k Z_\A\left(x_h \left(x_1^k \, \sh \, \widetilde{\reg}_\sh(w^\prime)\right)\right) \sigma_{Z_\A}\left(\frac{T^l}{l!}\right)\Bigg) \\
            & = \ev_0^\A \Bigg( \sum_{\substack{k+l=m \\ k,l\geq0}} (-1)^k Z_\A\left(x_h \left(x_1^k \, \sh \, \widetilde{\reg}_\sh(w^\prime)\right)\right) \sum_{j=0}^l\delta_j \frac{T^{l-j}}{(l-j)!}\Bigg) \\
            & = \sum_{\substack{k+l=m \\ k,l\geq0}} (-1)^k Z_\A\left(x_h \left(x_1^k \, \sh \, \widetilde{\reg}_\sh(w^\prime)\right)\right) \delta_l,
        \end{align*}
        where the first equality comes from \eqref{dist:lhs} and assumption of $Z_\A$, third one from from the composition \eqref{compositionZsh} which interprets $\overline{Z}_\A^\sh$ as a composition involving $\overline{\reg}_\sh^T$, which enables us to use expression \eqref{barregshTExpression} and the fifth equality comes from \eqref{sigma_dec}. This implies that
        \begin{equation*}
            \left( p_\ast^d (\Phi_{Z_\A \circ \rsh}) \Big| x_1^m w \right) = \Bigg(\exp\Bigg(\sum_{g^d=1} (\Phi_{Z_\A \circ \rsh}|x_g) x_1 \Bigg) i^\ast_d(\Phi_{Z_\A \circ \rsh}) \Big| x_1^m w \Bigg),
        \end{equation*}
        which proves that $\Phi_{Z_\A \circ \rsh} \in \DMRD(G)(\A)$.
        \item Assume that $Z_\A \notin \EDSD(G)(A)$. Therefore, there exists divisor $d$ of the order of $G$, $m \geq 0$ and $w \in \h_{G^d}$ that does not start with $x_1$ such that
        \[
            \overline{Z}_\A^\sh \circ p_\sharp^d(x_1^m w) \neq \sigma_{Z_\A} \circ \overline{Z}_\A^\sh \circ i^\sharp_d(x_1^m w).
        \]
        From Proposition \ref{prop:dist_iff_noeval}, it follows that
        \[
            \ev_0^\A \circ \overline{Z}_\A^\sh \circ p_\sharp^d(x_1^m w) \neq \ev_0^\A \circ \sigma_{Z_\A} \circ \overline{Z}_\A^\sh \circ i^\sharp_d(x_1^m w).
        \]
       On the other hand,
        \begin{align*}
            & \Bigg(\exp\Bigg(\sum_{g^d=1} (\Phi_{Z_\A \circ \rsh}|x_g) x_1 \Bigg) i^\ast_d(\Phi_{Z_\A \circ \rsh}) \Big| x_1^m w \Bigg) = \sum_{\substack{k+l=m \\ k,l\geq0}} \ev_0^\A \circ \overline{Z}_\A^\sh \circ i^\sharp_d(x_1^k w) \, \delta_l \\
            & = \sum_{\substack{k+l=m \\ k,l\geq0}} \ev_0^\A \left( \sum_{\substack{k_1+k_2=k \\ k_1,k_2\geq0}}\frac{(-1)^{k_1}}{k_2}Z_\A( x_h (x_1^{k_1} \, \sh \, \widetilde{\reg}_\sh(w^\prime)) ) T^{k_2}\right) \, \delta_l \\
            & = \sum_{\substack{k+l=m \\ k,l\geq0}} (-1)^k Z_\A( x_h (x_1^k \, \sh \, \widetilde{\reg}_\sh(w^\prime)) ) \, \delta_l = \ev_0^\A \circ \sigma_{Z_\A} \circ \overline{Z}_\A^\sh \circ i_d^\sharp (x_1^m w).
        \end{align*}
        Therefore, thanks to this equality, identity \eqref{dist:lhs} and the assumption on $Z_\A$, it follows that
        \[
            \left( p_\ast^d (\Phi_{Z_\A \circ \rsh}) \Big| x_1^m w \right) \neq \Bigg(\exp\Bigg(\sum_{g^d=1} (\Phi_{Z_\A \circ \rsh}|x_g) x_1 \Bigg) i^\ast_d(\Phi_{Z_\A \circ \rsh}) \Big| x_1^m w \Bigg), 
        \]
        which implies that $\Phi_{Z_\A \circ \rsh} \notin \DMRD(G)(\A)$.
    \end{caselist}
\end{proof}

\subsection{Proof of Theorem \ref{thm:introzhao}}
We now prove Theorem \ref{thm:introzhao} stated in the introduction.
For this, we establish the equivalence of Theorem \ref{thm:introzhao} and Theorem \ref{thm:zhaoconj} in Proposition \ref{conj_equiv_thm}. After that we prove Theorem \ref{thm:zhaoconj}. These results collectively yield a proof of Theorem \ref{thm:introzhao}.

\begin{theorem}\label{thm:zhaoconj}
    Let $Z_\A : \h^0_G \to \A$ be a $\Q$-linear map and $d$ a divisor of the order of $G$ such that:
    \begin{enumerate}[label=(\roman*), leftmargin=*]
        \item \label{conj_i} $Z_\A \in \EDS(G)(\A)$; 
        \item \label{conj_ii} $Z_\A \circ p_\sharp^d(x_h) = Z_\A \circ i^\sharp_d(x_h)$, for any $h \in G^d \smallsetminus \{1\}$;
        \item \label{conj_iii} $Z_\A \circ p_\sharp^d(x_{h_1} x_{h_2}) = Z_\A \circ i^\sharp_d(x_{h_1} x_{h_2})$, for any $h_1 \in G^d \smallsetminus \{1\}$ and $h_2 \in G^d$;
    \end{enumerate}
    where \ref{conj_ii} and \ref{conj_iii} are equalities in $\A$. Then, the following equality in $\A[T]$ holds 
    \begin{equation}\label{eq:distrrel}
         \overline{Z}_\A^\sh \circ p_\sharp^d(x_{h_1} x_{h_2}) = \sigma_{Z_\A} \circ \overline{Z}_\A^\sh \circ i^\sharp_d(x_{h_1} x_{h_2}),
   \end{equation}
    for any $h_1, h_2 \in G^d \sqcup \{0\}$.
\end{theorem}

\begin{proposition} \label{conj_equiv_thm}
    Theorem \ref{thm:introzhao} is equivalent to Theorem \ref{thm:zhaoconj}.
\end{proposition}
\begin{proof}
    From Theorem \ref{main_bijection_2}, it follows that:
    \begin{enumerate}[label=(\alph*), leftmargin=*]
        \item assumption \ref{conj_i} is equivalent to the extended double shuffle relations;
        \item assumption \ref{conj_ii} is equivalent to weight one finite distribution relations;
        \item assumption \ref{conj_iii} is equivalent to depth two finite distribution relations;
        \item identity \eqref{eq:distrrel} is equivalent to regularized distribution relations.
    \end{enumerate}
\end{proof}

\noindent We then dedicate the rest of this section entirely to the proof of Theorem \ref{thm:zhaoconj}.
First, let us introduce the following elements

\begin{definition}
    Let $d$ be a divisor of the order of $G$. For $h, h_2 \in G^d$, $h_1 \in G^d \smallsetminus \{1\}$, $g, g_1, g_2 \in G \smallsetminus \{1\}$ define\footnote{Here we use the notation $\mathrm{RDS}$ similarly to \cite{Zha10}, where extended double shuffle relations are called regularized double shuffle relations.} the following elements of $\h^0_G$:
    \begin{align*}
        \mathrm{FDT}_{d, 1}(h) &\displaystyle:= d \sum_{g^d=h} x_0 x_g - x_0 x_h, \\
        \mathrm{FDT}_{d, 2}(h_1,h_2) &\displaystyle:= \sum_{g_1^d=h_1} \sum_{g_2^d=h_2} x_{g_1} x_{g_2} - x_{h_1} x_{h_2}, \\
        \mathrm{FDS}(g_1, g_2) &:= x_0 x_{g_1g_2} + x_{g_1} x_{g_1g_2}+ x_{g_2} x_{g_1g_2} - x_{g_1} x_{g_2} - x_{g_2} x_{g_1}, \\
        \mathrm{RDS}(g) &:= x_0 x_g + x_g x_g - x_g x_1.
    \end{align*}
\end{definition}

\begin{lemma} \label{kerZA}
    Let $Z_\A : \h^0_G \to \A$ be a $\Q$-linear map and $d$ be a divisor of the order of $G$. For $h, h_2 \in G^d$, $h_1 \in G^d \smallsetminus \{1\}$, $g, g_1, g_2 \in G \smallsetminus \{1\}$, we have
    \begin{enumerate}[label=(\alph*), leftmargin=*, itemsep=2mm]
        \item \label{ker_and_FDT1} $\mathrm{FDT}_{d, 1}(h) \in \ker(Z_\A) \Longleftrightarrow Z_\A \circ p_\sharp^d(x_0 x_h) = Z_\A \circ i^\sharp_d(x_0 x_h)$;
        \item \label{ker_and_FDT2} $\mathrm{FDT}_{d, 2}(h_1, h_2) \in \ker(Z_\A) \Longleftrightarrow Z_\A \circ p_\sharp^d(x_{h_1} x_{h_2}) = Z_\A \circ i^\sharp_d(x_{h_1} x_{h_2})$;
        \item \label{ker_and_FDS} $\mathrm{FDS}(g_1, g_2) \in \ker(Z_\A) \Longleftrightarrow Z_\A \circ \q_G^{-1}(x_{g_1} \ast x_{g_2}) = Z_\A(x_{g_1} \, \sh \, x_{g_2})$;
        \item \label{ker_and_RDS} Assume $Z_\A : \h^0_{G,\sh} \to \A$ is an algebra homomorphism. Then
        \[
            \mathrm{RDS}(g) \in \ker(Z_\A) \Longleftrightarrow Z_{\A}^{\st}(x_1 \ast x_g) = Z_{\A}^{\st}(x_1) \, Z_{\A}^{\st}(x_g),
        \]
        where $Z_{\A}^{\st} = \rho^{-1}_{Z_\A} \circ Z_\A^\sh \circ \q_G^{-1} : \h^1_G \to \A[T]$.  
    \end{enumerate}
\end{lemma}
\begin{proof}
    We present a complete proof for statement \ref{ker_and_RDS}. The proofs for statements \ref{ker_and_FDT1}, \ref{ker_and_FDT2}, and \ref{ker_and_FDS} proceed by similar arguments with appropriate modifications. We have
    \begin{align*}
        Z_{\A}^{\st}(x_1) & = \rho^{-1}_{Z_\A} \circ Z_\A^\sh \circ \q_G^{-1}(x_1) = T, \\
        Z_{\A}^{\st}(x_g) & = \rho^{-1}_{Z_\A} \circ Z_\A^\sh \circ \q_G^{-1}(x_g) = Z_\A(x_g), \\
        Z_{\A}^{\st}(x_1 \ast x_g) & = \rho^{-1}_{Z_\A} \circ Z_\A^\sh \circ \q_G^{-1}(x_1 x_g + x_g x_1 + x_0 x_g) = \rho^{-1}_{Z_\A} \circ Z_\A^\sh (x_1 x_g + x_g x_g + x_0 x_g) \\
        & = \rho^{-1}_{Z_\A} \circ Z_\A^\sh (x_1 \sh x_g - x_g x_1 + x_g x_g + x_0 x_g) \\
        & = \rho^{-1}_{Z_\A}\left(Z_\A^\sh(x_1) Z_\A^\sh(x_g)\right) + Z_\A(x_0 x_g + x_g x_g - x_g x_1) \\
        & = T \, Z_\A(x_g) + Z_\A(x_0 x_g + x_g x_g - x_g x_1).
    \end{align*}
    Therefore,
    \begin{align*}
        Z_{\A}^{\st}(x_1 \ast x_g) - Z_{\A}^{\st}(x_1) \, Z_{\A}^{\st}(x_g) & = T \, Z_\A(x_g) + Z_\A(x_0 x_g + x_g x_g - x_g x_1) - T \, Z_\A(x_g) \\
        & = Z_\A(x_0 x_g + x_g x_g - x_g x_1). 
    \end{align*}
    Hence,
    \[
        Z_{\A}^{\st}(x_1 \ast x_g) - Z_{\A}^{\st}(x_1) \, Z_{\A}^{\st}(x_g) = Z_\A\left(\mathrm{RDS}_{Z_\A}(g)\right),
    \]
    thus proving that the vanishing of the left-hand side is equivalent to that of the right-hand side, which is the wanted result.
\end{proof}

For any divisor $d$ of the order of $G$ denote
\[
    K_d \df \{g \in G | g^d = 1\} = \ker(p^d),
\]
which is a subgroup of $G$ of order $d$.\newline
The following theorem is a generalization of \cite[Theorem 4.6]{Zha10}.
\begin{theorem}
    \label{thm:FDTd1}
    Let $d$ be a divisor of the order of $G$. For $h \in G^d$, we have (equality in $\h^0_G$)
    \begin{align*}
        \mathrm{FDT}_{d, 1}(h) =
        \left\{
        \begin{aligned}
            & \sum_{g_1, g_2 \in K_d \smallsetminus \{1\}} && \mathrm{FDS}(g_1, g_2) + 2 \sum_{g \in K_d \smallsetminus \{1\}} \mathrm{RDS}(g) & \text{ if } h=1 \\
            & \sum_{\substack{g_1^d=h \\ g_2 \in K_d \smallsetminus \{1\}}} && \mathrm{FDS}(g_1, g_2) + \sum_{g^d=h} \mathrm{RDS}(g) - \mathrm{RDS}(h) & \text{ otherwise} \\
            &&& + \mathrm{FDT}_{d, 2}(h,1) - \mathrm{FDT}_{d, 2}(h,h) & 
        \end{aligned}
        \right.
    \end{align*}
\end{theorem}
\begin{proof}
    We proceed by considering two distinct cases for $h$.
    \begin{caselist}
        \item If $h=1$, we have
        \begin{align*}
            & \sum_{g_1, g_2 \in K_d \smallsetminus \{1\}} \mathrm{FDS}(g_1, g_2) + 2 \sum_{g \in K_d \smallsetminus \{1\}} \mathrm{RDS}(g) \\
            &
            = \sum_{g_1, g_2 \in K_d \smallsetminus \{1\}} (x_0 x_{g_1g_2} + x_{g_1} x_{g_1g_2} + x_{g_2} x_{g_1g_2} - x_{g_1} x_{g_2} - x_{g_2} x_{g_1}) + 2 \sum_{g \in K_d \smallsetminus \{1\}} (x_0 x_g + x_g x_g - x_g x_1) \\
            & = \sum_{g_1, g_2 \in K_d \smallsetminus \{1\}} x_0 x_{g_1g_2} + \sum_{g_1, g_2 \in K_d \smallsetminus \{1\}} x_{g_1} x_{g_1g_2} + \sum_{g_1, g_2 \in K_d \smallsetminus \{1\}} x_{g_1} x_{g_1g_2} - \sum_{g_1, g_2 \in K_d \smallsetminus \{1\}} x_{g_1} x_{g_2} \\
            & - \sum_{g_1, g_2 \in K_d \smallsetminus \{1\}} x_{g_1} x_{g_2} + 2 \sum_{g \in K_d \smallsetminus \{1\}} x_0 x_g + 2 \sum_{g \in K_d \smallsetminus \{1\}} x_g x_g - 2 \sum_{g \in K_d \smallsetminus \{1\}} x_g x_1 \\
            & = \sum_{g_1, g_2 \in K_d \smallsetminus \{1\}} x_0 x_{g_1g_2} + 2 \sum_{g_1, g_2 \in K_d \smallsetminus \{1\}} x_{g_1} x_{g_1g_2} - 2 \sum_{g_1, g_2 \in K_d \smallsetminus \{1\}} x_{g_1} x_{g_2} + 2 \sum_{g \in K_d \smallsetminus \{1\}} x_0 x_g \\
            & + 2 \sum_{g \in K_d \smallsetminus \{1\}} x_g x_g - 2 \sum_{g \in K_d \smallsetminus \{1\}} x_g x_1 \\
            & = \sum_{g \in K_d \smallsetminus \{1\}} x_0 x_{g^2} + \sum_{\substack{g_1, g_2 \in K_d \smallsetminus \{1\} \\ g_2 \neq g_1}} x_0 x_{g_1g_2} + 2 \sum_{g \in K_d \smallsetminus \{1\}} x_g x_1 + 2 \sum_{\substack{g_1, g_2 \in K_d \smallsetminus \{1\} \\ g_2 \neq g_1^{-1}}} x_{g_1} x_{g_1g_2} \\
            & - 2 \sum_{g \in K_d \smallsetminus \{1\}} x_g x_g - 2 \sum_{\substack{g_1, g_2 \in K_d \smallsetminus \{1\} \\ g_2 \neq g_1}} x_{g_1} x_{g_2} + 2 \sum_{g \in K_d \smallsetminus \{1\}} x_0 x_g + 2 \sum_{g \in K_d \smallsetminus \{1\}} x_g x_g - 2 \sum_{g \in K_d \smallsetminus \{1\}} x_g x_1 \\
            & = \sum_{g \in K_d \smallsetminus \{1\}} x_0 x_{g^2} + \sum_{\substack{g_1, g_2 \in K_d \smallsetminus \{1\} \\ g_2 \neq g_1}} x_0 x_{g_1g_2} + 2 \sum_{\substack{g_1, g_2 \in K_d \smallsetminus \{1\} \\ g_2 \neq g_1^{-1}}} x_{g_1} x_{g_1g_2} - 2 \sum_{\substack{g_1, g_2 \in K_d \smallsetminus \{1\} \\ g_2 \neq g_1}} x_{g_1} x_{g_2} \\
            & + 2 \sum_{g \in K_d \smallsetminus \{1\}} x_0 x_g \\
            & = \sum_{g \in K_d \smallsetminus \{1\}} x_0 x_{g^2} + \sum_{g_1 \in K_d \smallsetminus \{1\} } \Big(\sum_{g_2 \in K_d} x_0 x_{g_2} -x_0x_{g_1^2} - x_0 x_{g_1}\Big) + 2 \sum_{\substack{g_1, g_2 \in K_d \smallsetminus \{1\} \\ g_2 \neq g_1}} x_{g_1} x_{g_2} \\
            & - 2 \sum_{\substack{g_1, g_2 \in K_d \smallsetminus \{1\} \\ g_2 \neq g_1}} x_{g_1} x_{g_2} + 2 \sum_{g \in K_d \smallsetminus \{1\}} x_0 x_g \\
            & = \sum_{g \in K_d \smallsetminus \{1\}} x_0 x_{g^2} + (d-1) \sum_{g \in K_d} x_0 x_g - \sum_{g \in K_d \smallsetminus \{1\}} x_0x_{g^2} - \sum_{g \in K_d \smallsetminus \{1\}} x_0 x_g + 2 \sum_{g \in K_d \smallsetminus \{1\}} x_0 x_g \\
            & = (d-1) \sum_{g \in K_d} x_0 x_g + \sum_{g \in K_d \smallsetminus \{1\}} x_0 x_g = (d-1) \sum_{g \in K_d} x_0 x_g + \sum_{g \in K_d} x_0 x_g - x_0 x_1 \\
            & = d \sum_{g \in K_d} x_0 x_g - x_0 x_1 = \mathrm{FDT}_{d, 1}(1),
        \end{align*}
        where the sixth equality comes from
        \[
            \sum_{g_2 \in K_d} x_0 x_{g_1 g_2} = \sum_{g_2 \in K_d} x_0 x_{g_2}
        \]
        for any $g_1 \in K_d$, and from
        \[
            \sum_{\substack{g_1, g_2 \in K_d \smallsetminus \{1\} \\ g_2 \neq g_1^{-1}}} x_{g_1} x_{g_1g_2} = \sum_{\substack{g_1, g_2 \in K_d \smallsetminus \{1\} \\ g_2 \neq g_1}} x_{g_1} x_{g_2}.
        \]
        \item If $h \neq 1$, we have
        \begin{align*}
            & \sum_{\substack{g_1^d=h \\ g_2 \in K_d \smallsetminus \{1\}}} \mathrm{FDS}(g_1, g_2) + \sum_{g^d=h} \mathrm{RDS}(g) - \mathrm{RDS}(h) + \mathrm{FDT}_{d, 2}(h,1) - \mathrm{FDT}_{d, 2}(h,h) \\
            & = \sum_{\substack{g_1^d=h \\ g_2 \in K_d \smallsetminus \{1\}}} (x_0 x_{g_1g_2} + x_{g_1} x_{g_1g_2} + x_{g_2} x_{g_1g_2} - x_{g_1} x_{g_2} - x_{g_2} x_{g_1}) + \sum_{g^d=h} (x_0 x_g + x_g x_g - x_g x_1) \\
            & - (x_0 x_h + x_h x_h - x_h x_1) + \Big(\sum_{\substack{g_1^d=h \\ g_2 \in K_d}} x_{g_1} x_{g_2} - x_h x_1\Big) - \Big(\sum_{g_1^d=g_2^d=h} x_{g_1} x_{g_2} - x_h x_h\Big) \\
            & = \sum_{g_2 \in K_d \smallsetminus \{1\}} \sum_{g_1^d=h} x_0 x_{g_1g_2} + \sum_{g_1^d=h} \sum_{g_2 \in K_d \smallsetminus \{1\}} x_{g_1} x_{g_1g_2} + \sum_{g_2 \in K_d \smallsetminus \{1\}} \sum_{g_1^d=h} x_{g_2} x_{g_1g_2} \\
            & - \sum_{g_1^d=h} \sum_{g_2 \in K_d \smallsetminus \{1\}} x_{g_1} x_{g_2} - \sum_{g_1^d=h} \sum_{g_2 \in K_d \smallsetminus \{1\}} x_{g_2} x_{g_1} + \sum_{g^d=h} x_0 x_g + \sum_{g^d=h} x_g x_g - \sum_{g^d=h} x_g x_1 \\ 
            & - x_0 x_h + \sum_{g_1^d=h} \sum_{g_2 \in K_d} x_{g_1} x_{g_2} - \sum_{g_1^d=h} \sum_{g_2^d=h} x_{g_1} x_{g_2} \\ 
            & = \sum_{g_2 \in K_d \smallsetminus \{1\}} \sum_{g_1^d=h} x_0 x_{g_1} + \sum_{g_1^d=h} \sum_{\substack{g_2^d=h \\ g_2 \neq g_1}} x_{g_1} x_{g_2} + \sum_{g_2 \in K_d \smallsetminus \{1\}} \sum_{g_1^d=h} x_{g_2} x_{g_1} \\
            & - \sum_{g_1^d=h} \sum_{g_2 \in K_d \smallsetminus \{1\}} x_{g_1} x_{g_2} - \sum_{g_1^d=h} \sum_{g_2 \in K_d \smallsetminus \{1\}} x_{g_2} x_{g_1} + \sum_{g^d=h} x_0 x_g + \sum_{g^d=h} x_g x_g - \sum_{g^d=h} x_g x_1 \\
            & - x_0 x_h + \sum_{g_1^d=h} \sum_{g_2 \in K_d} x_{g_1} x_{g_2} - \sum_{g_1^d=h} \sum_{g_2^d=h} x_{g_1} x_{g_2} \\ 
            & = (d-1) \sum_{g^d=h} x_0 x_g - \sum_{g^d=h} x_g x_g + \sum_{g^d=h} x_g x_1 + \sum_{g^d=h} x_0 x_g + \sum_{g^d=h} x_g x_g - \sum_{g^d=h} x_g x_1 - x_0 x_h \\ 
            & = d \sum_{g^d=h} x_0 x_g - x_0 x_h = \mathrm{FDT}_{d, 1}(h),
        \end{align*}
        where the third equality follows from
        \[
            \sum_{g_1^d=h} x_0 x_{g_1 g_2} = \sum_{g_1^d=h} x_0 x_{g_1},
        \]
        for any $g_2 \in K_d$; and from
        \[
            \sum_{g_2 \in K_d} x_{g_1} x_{g_1 g_2} = \sum_{g_2^d=h} x_{g_1} x_{g_2},
        \]
        for any $g_1$ such that $g_1^d=h$; and from
        \[
            \sum_{g_1^d=h} x_{g_2} x_{g_1g_2} = \sum_{g_1^d=h} x_{g_2} x_{g_1},
        \]
        for any $g_2 \in K_d$.
    \end{caselist}
\end{proof}

\begin{proof}[Proof of Theorem \ref{thm:zhaoconj}]
    Depending of the values of $h_1$ and $h_2$, we summarize all possible cases in the following table:
    \begin{table}[H]
    \begin{tabular}{|c|c|c|c|}
         \hline
         \diagbox{$h_1$}{$h_2$} & $0$ & $1$ & $G^d \smallsetminus \{1\}$ \\
         \hline
         $0$ & \ref{h1_0_h2_0} & \ref{h1_0_h2_1} & \ref{h1_0_h2_Gd1} \\
         \hline
         $1$ & \ref{h1_Gd_h2_0} & \ref{h1_1_h2_1} & \ref{h1_1_h2_Gd1} \\
         \hline
         $G^d \smallsetminus \{1\}$ & \ref{h1_Gd_h2_0} & \ref{h1_Gd1_h2_Gd} & \ref{h1_Gd1_h2_Gd} \\
         \hline
    \end{tabular}
    \end{table}
    \begin{caselist}
        \item \label{h1_Gd1_h2_Gd}$h_1 \in G^d \smallsetminus \{1\}$ and $h_2 \in G^d$. This follows immediately from assumption \ref{conj_iii}.
        \item \label{h1_1_h2_1}$h_1 = h_2 = 1$. This follows from \ref{Zpx1m_Zix1m} of the proof of Proposition \ref{prop:dist_iff_noeval} with $m=2$. 
        \item \label{h1_0_h2_0}$h_1 = h_2 = 0$. We have
        \begin{equation}
            \label{dist_at_x0_p}
            \overline{Z}_\A^\sh \circ p_\sharp^d(x_0) = d \overline{Z}_\A^\sh(x_0) = 0.   
        \end{equation}
        Then
        \[
            \overline{Z}_\A^\sh \circ p_\sharp^d(x_0^2) =  \overline{Z}_\A^\sh \circ p_\sharp^d\left(\frac{1}{2} x_0 \, \sh \, x_0\right) = \frac{1}{2} \overline{Z}_\A^\sh \circ p_\sharp^d(x_0)^2 = 0,   
        \]
        where the second equality follows from the fact that $\overline{Z}_\A^\sh \circ p_\sharp^d : \h^0{G^d, \sh} \to \A[T]$ is an algebra homomorphism and the last one from identity \eqref{dist_at_x0_p}.
        On the other hand,
        \begin{equation}
            \label{dist_at_x0_i}
            \overline{Z}_\A^\sh \circ i^\sharp_d(x_0) = \overline{Z}_\A^\sh(x_0) = 0. 
        \end{equation}
        Then
        \[
            \sigma_{Z_\A} \circ \overline{Z}_\A^\sh \circ i^\sharp_d(x_0^2) = \sigma_{Z_\A} \circ \overline{Z}_\A^\sh \circ i^\sharp_d\left(\frac{1}{2}x_0 \, \sh \, x_0\right) = \frac{1}{2} \sigma_{Z_\A} \left(\overline{Z}_\A^\sh \circ i^\sharp_d(x_0)^2\right) = 0, 
        \]
        where the second equality follows from the fact that $\overline{Z}_\A^\sh \circ i^\sharp_d : \h^0{G^d, \sh} \to \A[T]$ is an algebra homomorphism and the last one from identity \eqref{dist_at_x0_i}.
        \item \label{h1_0_h2_1}$h_1 = 0$ and $h_2 = 1$. From Theorem \ref{thm:FDTd1}, we have that
        \begin{equation}
            \label{FDTd11_indentity}
            Z_\A\left(\mathrm{FDT}_{Z_\A, d, 1}(1)\right) = \sum_{g_1, g_2 \in K_d \smallsetminus \{1\}} Z_\A\left(\mathrm{FDS}_{Z_\A}(g_1, g_2)\right) + 2 \sum_{g \in K_d \smallsetminus \{1\}} Z_\A\left(\mathrm{RDS}_{Z_\A}(g)\right).
        \end{equation}
        From assumption \ref{conj_i}, one may apply Lemma \ref{kerZA} \ref{ker_and_FDS} and \ref{ker_and_RDS} to conclude that the right hand side of equality \eqref{FDTd11_indentity} is equal to zero. The result then follows by applying Lemma \ref{kerZA} \ref{ker_and_FDT1}.
        \item \label{h1_0_h2_Gd1}$h_1 = 0$ and $h_2 = h \in G^d \smallsetminus \{1\}$. From Theorem \ref{thm:FDTd1}, we have that
        \begin{align}
            \label{FDTd1h_indentity}
            Z_\A\left(\mathrm{FDT}_{Z_\A, d, 1}(h)\right) & = \sum_{\substack{g_1^d=h \\ g_2 \in K_d \smallsetminus \{1\}}} Z_\A\left(\mathrm{FDS}_{Z_\A}(g_1, g_2)\right) + \sum_{g^d=h} Z_\A\left(\mathrm{RDS}_{Z_\A}(g)\right) \\
            & - Z_\A\left(\mathrm{RDS}_{Z_\A}(h)\right)
            + Z_\A\left(\mathrm{FDT}_{Z_\A, d, 2}(h,1)\right) - Z_\A\left(\mathrm{FDT}_{Z_\A, d, 2}(h,h)\right). \notag
        \end{align}
        From assumption \ref{conj_i} (resp. \ref{conj_iii}), one may apply Lemma \ref{kerZA} \ref{ker_and_FDS} and \ref{ker_and_RDS} (resp. \ref{ker_and_FDT2}) to conclude that the right hand side of equality \eqref{FDTd1h_indentity} is equal to zero. The result then follows by applying Lemma \ref{kerZA} \ref{ker_and_FDT1}.
        \item \label{h1_Gd_h2_0}$h_1 = h \in G^d$, $h_2 = 0$. We have
        \begin{align*}
            \overline{Z}_\A^\sh \circ p_\sharp^d(x_h x_0) & = \overline{Z}_\A^\sh \circ p_\sharp^d(x_h \, \sh \, x_0 - x_0 x_h) = \overline{Z}_\A^\sh \circ p_\sharp^d(x_h \, \sh \, x_0) - \overline{Z}_\A^\sh \circ p_\sharp^d(x_0 x_h) \\
            & = \overline{Z}_\A^\sh \circ p_\sharp^d(x_h) \, \overline{Z}_\A^\sh \circ p_\sharp^d(x_0) - Z_\A \circ p_\sharp^d(x_0 x_h) = - Z_\A \circ p_\sharp^d(x_0 x_h) \\
            & = - Z_\A \circ i_d^\sharp(x_0 x_h), 
        \end{align*}
        where the second to last equality follows from equality \eqref{dist_at_x0_p} and the last one from the identity obtained in \ref{h1_0_h2_1} ($h=1$) and \ref{h1_0_h2_Gd1} ($h\neq1$). On the other hand,
        \begin{align*}
            \sigma_{Z_\A} \circ \overline{Z}_\A^\sh \circ i^\sharp_d(x_h x_0) & = \sigma_{Z_\A} \circ \overline{Z}_\A^\sh \circ i^\sharp_d(x_h \, \sh \, x_0 - x_0 x_h) \\
            & = \sigma_{Z_\A}\left(\overline{Z}_\A^\sh \circ i^\sharp_d(x_h \, \sh \, x_0) - \overline{Z}_\A^\sh \circ i^\sharp_d(x_0 x_h)\right) \\
            & = \sigma_{Z_\A}\left(\overline{Z}_\A^\sh \circ i^\sharp_d(x_h) \, \overline{Z}_\A^\sh \circ i^\sharp_d(x_0) - Z_\A^\sh \circ i^\sharp_d(x_0 x_h)\right) \\
            & = - Z_\A \circ i_d^\sharp(x_0 x_h),
        \end{align*}
        where the second to last equality follows from equality \eqref{dist_at_x0_i} and the last one from the fact that $x_0 x_h \in \h^0_{G^d}$ and $\sigma_{Z_\A}(1)=1$.
        \item \label{h1_1_h2_Gd1}$h_1 = 1$, $h_2 = h \in G^d \smallsetminus \{1\}$. We have
         \begin{align*}
            \overline{Z}_\A^\sh \circ p_\sharp^d(x_1 x_h) & = \overline{Z}_\A^\sh \circ p_\sharp^d(x_1 \, \sh \, x_h - x_h x_1) = \overline{Z}_\A^\sh \circ p_\sharp^d(x_1 \, \sh \, x_h) - \overline{Z}_\A^\sh \circ p_\sharp^d(x_h x_1) \\
            & = \overline{Z}_\A^\sh \circ p_\sharp^d(x_1) \, \overline{Z}_\A^\sh \circ p_\sharp^d(x_h) - Z_\A \circ p_\sharp^d(x_h x_1) \\
            & = \sigma_{Z_\A} \circ \overline{Z}_\A^\sh \circ i^\sharp_d(x_1) \, Z_\A^\sh \circ p_\sharp^d(x_h) - Z_\A \circ i^\sharp_d(x_h x_1) \\
            & = \sigma_{Z_\A} \left(\overline{Z}_\A^\sh \circ i^\sharp_d(x_1)\right) \, Z_\A^\sh \circ i^\sharp_d(x_h) - \sigma_{Z_\A} \left(Z_\A \circ i^\sharp_d(x_h x_1)\right) \\
            & = \sigma_{Z_\A}\left(\overline{Z}_\A^\sh \circ i^\sharp_d(x_1) \, Z_\A^\sh \circ i^\sharp_d(x_h) - Z_\A \circ i^\sharp_d(x_h x_1)\right) \\
            & = \sigma_{Z_\A}\left(\overline{Z}_\A^\sh \circ i^\sharp_d(x_1 \, \sh \, x_h - x_h x_1)\right) = \sigma_{Z_\A} \circ \overline{Z}_\A^\sh \circ i^\sharp_d(x_1 x_h), 
        \end{align*}
        where the third equality follows from the fact that $\overline{Z}_\A^\sh \circ p_\sharp^d : \h^0_{G^d, \sh} \to \A[T]$ is an algebra homomorphism; the fourth one from \ref{Zpx1m_Zix1m} of the proof of Proposition \ref{prop:dist_iff_noeval} and assumption \ref{conj_iii}, the fifth one from assumption \ref{conj_ii} and the fact that $\sigma_{Z_\A}(1)=1$, the sixth one from $\A$-linearity of $\sigma_{Z_\A}$ and the seventh one from the fact that $\overline{Z}_\A^\sh \circ i^\sharp_d : \h^0_{G^d, \sh} \to \A[T]$ is an algebra homomorphism.
    \end{caselist}
\end{proof}

\begin{remark}
    Applying the reasoning analogous to that in Proposition \ref{conj_equiv_thm}, one can verify that the statement in \ref{h1_1_h2_Gd1} is equivalent to \cite[Conjecture 7.1]{Zha10}. Consequently, the proof provided in \ref{h1_1_h2_Gd1} also establishes this conjecture.
\end{remark}

    \appendix
    \section{Quasi-shuffle products}

We will consider now a more general point of view for all the products considered before. For this, suppose that we have a set $\LL$, called the set of \emph{letters}, and a commutative and associative product $\diamond$ in $\LL$. Extending this product bi-linearly to $\Q \LL$ gives a commutative non-unital $\Q$-algebra $(\Q \LL,  \diamond)$. We will be interested in $\Q$-linear combinations of multiple letters of $\LL$, i.e., in elements of $\QL$. Here and in the following we call monic monomials in $\QL$ \emph{words}. Moreover, we call the degree of this monomial the \emph{length} of the word. 
	
\begin{definition}
Define the \emph{quasi-shuffle product}  $\qsh$ on $\QL$ as the $\Q$-bilinear product which satisfies $1 \qsh w = w \qsh 1 = w$ for any word $w\in \QL$ and
    \begin{align*}
        a w \qsh b v = a (w \qsh b v) + b (a w \qsh v) + (a \diamond b) (w \qsh  v) 
    \end{align*}
    for any letters $a,b \in \LL$ and words $w, v \in \QL$. 
\end{definition}
This gives a commutative $\Q$-algebra $(\QL,\qsh)$ as shown in~\cite{HI}.  Moreover, one can equip this algebra with the structure of a Hopf algebra \cite[Section 3]{HI}, where the coproduct is given for $w \in \QL$ by the deconcatenation coproduct
\begin{align*}
    \Delta(w) = \sum_{uv = w} u \otimes v\,.
\end{align*}

Now we consider the completed dual Hopf algebra of $(\QL,\qsh,\Delta)$. For this, we need to give a restriction on $\diamond$, that is, we assume that any $a \in \LL$ just appears in $b \diamond c$ for finitely many $b,c \in \LL$. More precisely, define for $a,b,c \in \LL$ the coefficients $\lambda^{a}_{b,c} \in \Q$ by $$b \diamond c = \sum_{a \in \LL} \lambda^{a}_{b,c} a.$$ In the following, we will assume that for any $a\in \LL$ there are just finitely many $b,c \in \LL$ with  $\lambda^{a}_{b,c} \neq 0$.

As a consequence of Lemma \ref{lem:coproductonp} below, the dual completed Hopf algebra of $(\QL,\qsh,\Delta)$ is then given by $(\ALL, \operatorname{conc}, \widehat{\Delta}_{\qsh})$, where $\operatorname{conc}$ is the concatenation product and $\widehat{\Delta}_{\qsh} : \ALL \to \ALL^{\hat{\otimes} 2}$ is the unique $\Q$-algebra homomorphism such that 
\begin{equation*}
    \widehat{\Delta}_{\qsh}(a) = a \otimes 1 + 1 \otimes a + \sum_{ b,c \in \LL} \lambda^{a}_{b,c} \,\,b \otimes c
\end{equation*}
for any $a \in \LL$. Since we just consider $\diamond$ such that for any $a\in \LL$ the coefficients $ \lambda^{a}_{b,c}$ are zero for almost all $b,c \in \LL$ and therefore $\widehat{\Delta}_{\qsh}$ is well-defined.

\begin{definition}
    \label{def:scalarproduct}
    Define the pairing
    \begin{align*}
        \ALL \otimes \QL &\longrightarrow \A\\
        \psi \otimes w &\longmapsto  (\psi  \mid w),
    \end{align*}
    where, as before, $(\psi  \mid w)$ denotes the coefficient of $w$ in $\psi$.
\end{definition}

\begin{lemma}\label{lem:coproductonp}
    For any $\Phi \in \ALL$ we have
    \begin{align*}
        \widehat{\Delta}_{\qsh}(\Phi) = \sum_{u,v \in \LL^*} (\Phi \mid u \qsh v) u \otimes v.
    \end{align*}
\end{lemma}
\begin{proof} Notice that the statement is equivalent to $(\widehat{\Delta}_{\qsh}(\Phi) \mid  u \otimes v) =  (\Phi \mid u \qsh v)$ for all $u,v \in \LL^*$. Due to linearity, it suffices to show \begin{align}\label{eq:dualwords}
    (\widehat{\Delta}_{\qsh}(w) \mid  u \otimes v) =  (w \mid u \qsh v)
\end{align} for words $w,u,v \in \LL^*$. We will perform induction on the length of $w$. In the length case $1$, that is, when $w=a \in \LL$ we have
\begin{align*}
    \widehat{\Delta}_{\qsh}(a) = a \otimes 1 + 1 \otimes a + \sum_{ a_1,a_2 \in \LL} \lambda^{a}_{a_1,a_2} \,\,a_1 \otimes a_2.
\end{align*}
By the definition of $u \qsh v$ we see that $u=a, v=1$ or $u=1,v=a$ or $u=a_1, v=a_2$ (with multiplicity $\lambda^{a}_{a_1,a_2}$) are exactly cases where $a$ can appear as a coefficient of $u \qsh v$ and therefore $(\widehat{\Delta}_{\qsh}(a) \mid  u \otimes v) =  (a \mid u \qsh v)$.

Next, notice that \eqref{eq:dualwords} is trivial for any $w$ when $u$ or $v$ is the empty word. Therefore, we can restrict ourselves to the case $w=aw'$, $u=bu'$ and $v=cv'$, for letters $a,b,c \in \LL$ and words $w',u',v' \in \LL^*$. Then we want to show that 
\begin{align}\label{eq:deltaawbucv}
    (\widehat{\Delta}_{\qsh}(aw') \mid  bu' \otimes cv') =  (aw' \mid bu' \qsh cv').
\end{align} 
For the left-hand side of \eqref{eq:deltaawbucv} we have
\begin{align*}
(\widehat{\Delta}_{\qsh}(aw') \mid  bu' \otimes cv') =\,\,  &(\widehat{\Delta}_{\qsh}(a)\widehat{\Delta}_{\qsh}(w') \mid  bu' \otimes cv')\\
=\,\, &( (a \otimes 1)\widehat{\Delta}_{\qsh}(w')\mid  bu' \otimes cv') + ((1 \otimes a)\widehat{\Delta}_{\qsh}(w')\mid  bu' \otimes c'v) \\
&+ \sum_{ a_1,a_2 \in \LL} \lambda^{a}_{a_1,a_2} \,\,( (a_1 \otimes a_2)\widehat{\Delta}_{\qsh}(w')\mid  bu' \otimes cv') \\
=\,\, &\delta_{a,b} (\widehat{\Delta}_{\qsh}(w') \mid  u' \otimes cv')  + \delta_{a,c} (\widehat{\Delta}_{\qsh}(w') \mid  bu' \otimes v')  \\
&+ \sum_{ a_1,a_2 \in \LL} \lambda^{a}_{a_1,a_2} \delta_{a_1,b} \delta_{a_2,c}\, ( \widehat{\Delta}_{\qsh}(w')\mid  u' \otimes v')\\
=\,\, &\delta_{a,b} (\widehat{\Delta}_{\qsh}(w') \mid  u' \otimes cv')  + \delta_{a,c} (\widehat{\Delta}_{\qsh}(w') \mid  bu' \otimes v')  \\
&+\lambda^{a}_{b,c} \, ( \widehat{\Delta}_{\qsh}(w')\mid  u' \otimes v').
\end{align*}
Here, $\delta_{a,b}$ denotes the Kronecker delta. Using the definition of $bu' \qsh cv'$, the right-hand side of \eqref{eq:deltaawbucv} is given by
\begin{align*}
(aw' \mid bu' \qsh cv') &= (aw' \mid b(u' \qsh cv')) + (aw' \mid c(bu'\qsh v')) + (aw' \mid (b \diamond c)(u'\qsh v')) \\
&= \delta_{a,b} (w' \mid u' \qsh cv') + \delta_{a,c} (w' \mid bu'\qsh v') + \lambda^a_{b,c} (w' \mid u' \qsh v').
\end{align*}
Equation \eqref{eq:deltaawbucv} then follows from the induction hypothesis. 
\end{proof}

As before, denote the set grouplike elements of $\ALL$ for the coproduct $\widehat{\Delta}_{\qsh}$ by
\begin{equation*}
    \G(\ALL) = \{ \Phi \in \ALL^{\times} \, | \, \widehat{\Delta}_{\qsh}(\Phi) = \Phi \otimes \Phi \}.
\end{equation*}

Given a $\Q$-linear map $\varphi: \QL \rightarrow \A$ we define the following element in $\ALL$
\begin{align*}
    \Phi_\varphi \df \sum_{w \in \LL^*} \varphi(w) w\,.
\end{align*}

\begin{proposition} \label{alghom_iff_grplk}
    The following two statements are equivalent
    \begin{enumerate}[label=(\alph*), leftmargin=*]
        \item The map $\varphi: (\QL,\qsh) \rightarrow \A$ is an $\Q$-algebra homomorphism.
        \item We have $\Phi_\varphi \in \G(\ALL)$.
    \end{enumerate}
\end{proposition}
\begin{proof} First notice that Lemma \ref{lem:coproductonp} gives
\begin{align*}
    \widehat{\Delta}_{\qsh}( \Phi_\varphi ) = \sum_{w,v \in \LL^*} \underbrace{(\Phi_\varphi | w \qsh v)}_{ = \,\varphi(w \qsh v)} w \otimes v.
\end{align*}
By direct calculation we have
\begin{align*}
\Phi_\varphi \otimes \Phi_\varphi &= \left( \sum_{w \in \LL^*} \varphi(w) w \right) \otimes \left( \sum_{v \in \LL^*} \varphi(v) v \right) = \sum_{w,v \in \LL^*} \varphi(w)\varphi(v)  \, w \otimes v.
\end{align*}
Therefore, $\widehat{\Delta}_{\qsh}( \Phi_\varphi ) = \Phi_\varphi \otimes \Phi_\varphi$ if and only if $\varphi(w \qsh v) = \varphi(w) \varphi(v)$ for all $w,v \in \LL^*$.
\end{proof}

    \section{Functoriality results}

In this appendix, we revisit the functors $(-)^\ast$ and $(-)_\ast$ defined in \cite{Rac02}. This overview aims to clarify key aspects and highlight some important properties of these functors. This will provide a useful foundation for the definition of their respective dual counterparts $(-)^\sharp$ and $(-)_\sharp$ that we introduce in this paper. \newline
Throughout this appendix we will denote by $\mathsf{FinAb}$ the category of finite abelian groups and $\A{\text-}\mathsf{Alg}$ the category of $\A$-algebras. \newline
Proposition-Definition \ref{star_functors} and Lemmas \ref{star_Hopf} and \ref{ast_comm} below are mentioned in \cite[§2.5.3]{Rac02} without proof.

\begin{propdef}[{\cite[§2.5.3]{Rac02}}] \ 
    \begin{enumerate}[label=(\alph*), leftmargin=*]
        \item \label{upper_star} Define the mapping $(-)^\ast : \mathsf{FinAb} \to \A{\text-}\mathsf{Alg}$ as follows:
        \begin{enumerate}[label=(\roman*), leftmargin=*]
            \item To each group $G$, associate the algebra $\A\langle\langle X_G\rangle\rangle$;
            \item To each group homomorphism $\phi : G_1 \to G_2$, associate the algebra homomorphism $\phi^\ast : \A\langle\langle X_{G_2}\rangle\rangle \to \A\langle\langle X_{G_1}\rangle\rangle$ given by
            \[
                x_0 \mapsto x_0; \qquad x_h \mapsto \sum_{g\in\phi^{-1}(\{h\})} x_g \quad (h \in G_2). 
            \]  
        \end{enumerate}
        Then the mapping $(-)^\ast$ is a contravariant functor.
        \item \label{lower_star} Define the mapping $(-)_\ast : \mathsf{FinAb} \to \A{\text-}\mathsf{Alg}$ as follows:
        \begin{enumerate}[label=(\roman*), leftmargin=*]
            \item To each group $G$, associate the algebra $\A\langle\langle X_G\rangle\rangle$;
            \item To each group homomorphism $\phi : G_1 \to G_2$, associate the algebra homomorphism $\phi_\ast : \A\langle\langle X_{G_1}\rangle\rangle \to \A\langle\langle X_{G_2}\rangle\rangle$ given by
            \[
                x_0 \mapsto d \, x_0 \quad (d=|\ker(\phi)|); \qquad x_g \mapsto x_{\phi(g)} \quad (g \in G_1). 
            \]  
        \end{enumerate}
        Then the mapping $(-)_\ast$ is a covariant functor.
    \end{enumerate}
    \label{star_functors}
\end{propdef}
\begin{proof} \ 
    \begin{enumerate}[label=(\alph*), leftmargin=*]
        \item This can be established through direct verification, utilizing the contravariance of the power set functor $P^\ast$. Recall that $P^\ast$ maps each set $E$ to its power set $\mathcal{P}(E)$ and each map $f : E \to F$ to the map $P^\ast(f) : \mathcal{P}(F) \to \mathcal{P}(E)$ defined by $P^\ast(f)(U) = f^{-1}(U)$ for $U \subseteq F$.
        \item This can be established through direct verification, utilizing the covariance of the power set functor $P_\ast$. Recall that $P_\ast$ maps each set $E$ to its power set $\mathcal{P}(E)$ and each map $f : E \to F$ to the map $P_\ast(f) : \mathcal{P}(E) \to \mathcal{P}(F)$ defined by $P_\ast(f)(S) = f(S)$ for $S \subseteq E$.
    \end{enumerate}
\end{proof}

\begin{lemma}
    \label{star_Hopf}
    Let $\phi : G_1 \to G_2$ be an arrow in $\mathsf{FinAb}$. Then, the algebra homomorphism $\phi^\ast : \A\langle\langle X_{G_2}\rangle\rangle \to \A\langle\langle X_{G_1}\rangle\rangle$ (resp. $\phi_\ast : \A\langle\langle X_{G_1}\rangle\rangle \to \A\langle\langle X_{G_2}\rangle\rangle$):
    \begin{enumerate}[label=(\alph*), leftmargin=*]
        \item \label{star_Hopf_sh}is a Hopf algebra homomorphism $\phi^\ast : (\A\langle\langle X_{G_2}\rangle\rangle, \widehat{\Delta}_{G_2,\sh}) \to (\A\langle\langle X_{G_1}\rangle\rangle, \widehat{\Delta}_{G_1,\sh})$ (resp. $\phi_\ast : (\A\langle\langle X_{G_1}\rangle\rangle, \widehat{\Delta}_{G_1,\sh}) \to (\A\langle\langle X_{G_2}\rangle\rangle, \widehat{\Delta}_{G_2,\sh})$);
        \item \label{star_Hopf_st}restricts to an algebra homomorphism $\phi^\ast : \A\langle\langle Y_{G_2}\rangle\rangle \to \A\langle\langle Y_{G_1}\rangle\rangle$ (resp. $\phi_\ast : \A\langle\langle Y_{G_1}\rangle\rangle \to \A\langle\langle Y_{G_2}\rangle\rangle$), which is a Hopf algebra homomorphism $(\A\langle\langle Y_{G_2}\rangle\rangle, \widehat{\Delta}_{G_2,\st}) \to (\A\langle\langle Y_{G_1}\rangle\rangle, \widehat{\Delta}_{G_1,\st})$ (resp. $(\A\langle\langle Y_{G_1}\rangle\rangle, \widehat{\Delta}_{G_1,\st}) \to (\A\langle\langle Y_{G_2}\rangle\rangle, \widehat{\Delta}_{G_2,\st})$).
    \end{enumerate}
\end{lemma}
\begin{proof} \ 
    \begin{enumerate}[label=(\alph*), leftmargin=*]
        \item The fact that $\widehat{\Delta}_{G_1,\sh} \circ \phi^\ast = (\phi^\ast)^{\otimes 2} \circ \widehat{\Delta}_{G_2,\sh}$ (resp. $\widehat{\Delta}_{G_2,\sh} \circ \phi_\ast = (\phi_\ast)^{\otimes 2} \circ \widehat{\Delta}_{G_1,\sh}$) can be established through direct verification on the generators.
        \item This follows from the fact that $\phi^\ast(x_0^{n-1} x_h) = \sum_{g \in \phi^{-1}(\{h\})} x_0^{n-1} x_g$ for any $(n, h) \in \Z_{>0} \times G_2$ (resp. $\phi_\ast(x_0^{n-1} x_g) = d^{n-1} x_0^{n-1} x_{\phi(g)}$ for any $(n, g) \in \Z_{>0} \times G_1$). From this, one can establish through a direct verification that $\widehat{\Delta}_{G_1,\st} \circ \phi^\ast = (\phi^\ast)^{\otimes 2} \circ \widehat{\Delta}_{G_2,\st}$ (resp. $\widehat{\Delta}_{G_2,\st} \circ \phi_\ast = (\phi_\ast)^{\otimes 2} \circ \widehat{\Delta}_{G_1,\st}$)  
    \end{enumerate}
\end{proof}

\begin{lemma}
    \label{ast_comm}
    Let $\phi : G_1 \to G_2$ be an arrow in $\mathsf{FinAb}$.
    \begin{enumerate}[label=(\alph*), leftmargin=*]
        \item We have an equality of $\A$-module homomorphisms $\A\langle\langle X_{G_2}\rangle\rangle \to \A\langle\langle Y_{G_1}\rangle\rangle$
        \[
            \pi_{Y_{G_1}} \circ \phi^\ast = \phi^\ast \circ \pi_{Y_{G_2}}
        \]
        and an equality of $\A$-module homomorphisms $\A\langle\langle X_{G_1}\rangle\rangle \to \A\langle\langle Y_{G_2}\rangle\rangle$
        \[
            \pi_{Y_{G_2}} \circ \phi_\ast = \phi_\ast \circ \pi_{Y_{G_1}}.
        \] 
        \item \label{qphiast_comm} We have an equality of $\A$-module homomorphisms $\A\langle\langle X_{G_2}\rangle\rangle \to \A\langle\langle X_{G_1}\rangle\rangle$
        \[
            \qq_{G_1} \circ \phi^\ast = \phi^\ast \circ \qq_{G_2}
        \]
        and an equality of $\A$-module homomorphisms $\A\langle\langle X_{G_1}\rangle\rangle \to \A\langle\langle X_{G_2}\rangle\rangle$
        \[
            \qq_{G_2} \circ \phi_\ast = \phi_\ast \circ \qq_{G_1}.
        \] 
    \end{enumerate}
\end{lemma}
\begin{proof} \ 
    \begin{enumerate}[label=(\alph*), leftmargin=*]
        \item This follows from the direct sum decomposition $\A\langle\langle X_{G_i}\rangle\rangle = \A\langle\langle Y_{G_i}\rangle\rangle \oplus \A\langle\langle X_{G_i}\rangle\rangle x_0$ ($i \in \{1, 2\}$) and the fact that $\phi^\ast$ (resp. $\phi_\ast$) restricts to $\A\langle\langle Y_{G_2}\rangle\rangle \to \A\langle\langle Y_{G_1}\rangle\rangle$ (resp. $\A\langle\langle Y_{G_1}\rangle\rangle \to \A\langle\langle Y_{G_2}\rangle\rangle$), thanks to Lemma \ref{star_Hopf} \ref{star_Hopf_st}.
        \item It is enough to prove this equality on a word $x_0^{n_1-1}x_{h_1} \cdots x_0^{n_r-1}x_{h_r} x_0^{n_{r+1}-1}$ of $\A\langle\langle X_{G_2}\rangle\rangle$ (resp. $x_0^{n_1-1}x_{g_1} \cdots x_0^{n_r-1}x_{g_r} x_0^{n_{r+1}-1}$ of $\A\langle\langle X_{G_1}\rangle\rangle$). We have
        \begin{align*}
            &\mbox{\small$\displaystyle\qq_{G_1} \circ \phi^\ast(x_0^{n_1-1}x_{h_1} \cdots x_0^{n_r-1}x_{h_r} x_0^{n_{r+1}-1}) = \sum_{\substack{\phi(g_i)=h_i \\ 1\leq i\leq r}} \qq_{G_1}(x_0^{n_1-1}x_{g_1} \cdots x_0^{n_r-1}x_{g_r} x_0^{n_{r+1}-1})$} \\
            & = \sum_{\substack{\phi(g_i)=h_i \\ 1\leq i\leq r}} x_0^{n_1-1}x_{g_1} x_0^{n_2-1}x_{g_1^{-1}g_2} \cdots x_0^{n_r-1}x_{g_{r-1}^{-1}g_r} x_0^{n_{r+1}-1} \\
            & = \sum_{\phi(g_1^\prime)=h_1} \sum_{\substack{\phi(g_i^\prime)=h_{i-1}^{-1} h_i \\ 2\leq i\leq r}} x_0^{n_1-1}x_{g_1^\prime} x_0^{n_2-1}x_{g_2^\prime} \cdots x_0^{n_r-1}x_{g_r^\prime} x_0^{n_{r+1}-1} \\
            & = \phi^\ast\left(x_0^{n_1-1}x_{h_1} x_0^{n_2-1}x_{h_1^{-1}h_2} \cdots x_0^{n_r-1}x_{h_{r-1}^{-1} h_r} x_0^{n_{r+1}-1}\right) \\
            & = \phi^\ast \circ \qq_{G_2}(x_0^{n_1-1}x_{h_1} \cdots x_0^{n_r-1}x_{h_r} x_0^{n_{r+1}-1}), 
        \end{align*}
        where the third equality follows by applying the change of variables
        \[
            g_1^\prime = g_1, \quad g_2^\prime = g_1^{-1} g_2, \quad \dots, \quad g_r^\prime = g_{r-1}^{-1} g_r.
        \]
       For the second identity we get
        \begin{align*}
            &\mbox{\small$\displaystyle\qq_{G_2} \circ \phi_\ast(x_0^{n_1-1}x_{g_1} \cdots x_0^{n_r-1}x_{g_r} x_0^{n_{r+1}-1}) = \qq_{G_1}(x_0^{n_1-1}x_{\phi(g_1)} \cdots x_0^{n_r-1}x_{\phi(g_r)} x_0^{n_{r+1}-1})$} \\
            & = x_0^{n_1-1}x_{\phi(g_1)} x_0^{n_2-1}x_{\phi(g_1)^{-1}\phi(g_2)} \cdots x_0^{n_r-1}x_{\phi(g_{r-1})^{-1}\phi(g_r)} x_0^{n_{r+1}-1} \\
            & = x_0^{n_1-1}x_{\phi(g_1)} x_0^{n_2-1}x_{\phi(g_1^{-1}g_2)} \cdots x_0^{n_r-1}x_{\phi(g_{r-1}^{-1}g_r)} x_0^{n_{r+1}-1} \\
            & = \phi_\ast\left(x_0^{n_1-1}x_{g_1} x_0^{n_2-1}x_{g_1^{-1}g_2} \cdots x_0^{n_r-1}x_{g_{r-1}^{-1} g_r} x_0^{n_{r+1}-1}\right) \\
            & = \phi_\ast \circ \qq_{G_1}(x_0^{n_1-1}x_{g_1} \cdots x_0^{n_r-1}x_{g_r} x_0^{n_{r+1}-1}), 
        \end{align*}
        where the third equality follows from the fact that $\phi$ is a group homomorphism.
    \end{enumerate}
\end{proof}

We now define the dual counterparts of the functors $(-)^\ast$ and $(-)_\ast$. 
\begin{propdef} \ 
    \begin{enumerate}[label=(\alph*), leftmargin=*]
        \item \label{upper_sharp} Define a mapping $(-)^\sharp : \mathsf{FinAb} \to \A{\text-}\mathsf{Alg}$ as follows:
        \begin{enumerate}[label=(\roman*), leftmargin=*]
            \item To each group $G$, associate the algebra $\h_G$;
            \item To each group homomorphism $\phi : G_1 \to G_2$, associate the algebra homorphism $\phi^\sharp : \h_{G_1} \to \h_{G_2}$ given by
            \[
                x_0 \mapsto x_0; \qquad x_g \mapsto x_{\phi(g)} \quad (g \in G_1). 
            \]  
        \end{enumerate}
        Then the mapping $(-)^\sharp$ is a covariant functor.
        \item \label{lower_sharp} Define a mapping $(-)_\sharp : \mathsf{FinAb} \to \A{\text-}\mathsf{Alg}$ as follows:
        \begin{enumerate}[label=(\roman*), leftmargin=*]
            \item To each group $G$, associate the algebra $\h_G$;
            \item To each group homomorphism $\phi : G_1 \to G_2$, associate the algebra homorphism $\phi_\sharp : \h_{G_2} \to \h_{G_1}$ given by
            \[
                x_0 \mapsto d \, x_0 \quad (d=|\ker(\phi)|); \qquad x_h \mapsto \sum_{g\in\phi^{-1}(\{h\})} x_g \quad (h \in G_2). 
            \]  
        \end{enumerate}
        Then the mapping $(-)_\sharp$ is a contravariant functor.
    \end{enumerate}
    \label{sharp_functors}
\end{propdef}
\begin{proof}
    This can be proven by applying the same reasoning used in the proof of Proposition-Definition \ref{star_functors}.
\end{proof}

Let $\phi : G_1 \to G_2$ be an arrow in $\mathsf{FinAb}$. For $i \in \{1, 2\}$, consider the pairing
\[
    (-, -)_i : \A\langle\langle X_{G_i}\rangle\rangle \otimes \h_{G_i} \to \A,
\]
from Definition \ref{def:scalarproduct} for the case $\mathcal{L}=X_{G_i}$.

\begin{lemma}
    \label{lem:astsharp}
    Let $\phi : G_1 \to G_2$ be an arrow in $\mathsf{FinAb}$. We have
    \begin{enumerate}[label=(\alph*), leftmargin=*]
        \item $(\phi^\ast(S_2), P_1)_1 = (S_2, \phi^\sharp(P_1))_2$, for $S_2 \in \A\langle\langle X_{G_2}\rangle\rangle$ and $P_1 \in \h_{G_1}$.
        \item $(\phi_\ast(S_1), P_2)_2 = (S_1, \phi_\sharp(P_2))_1$, for $S_1 \in \A\langle\langle X_{G_1}\rangle\rangle$ and $P_2 \in \h_{G_2}$.
    \end{enumerate}
\end{lemma}
\begin{proof} \ 
    \begin{enumerate}[label=(\alph*), leftmargin=*]
        \item It is enough to show this equality for $P_1 = x_0^{n_1-1} x_{g_1} \cdots x_0^{n_r-1} x_{g_r} x_0^{n_{r+1}-1} \in X_{G_1}^\ast$, due to linearity. We have $\phi^\sharp(P_1) = x_0^{n_1-1} x_{\phi(g_1)} \cdots x_0^{n_r-1} x_{\phi(g_r)} x_0^{n_{r+1}-1}$. Therefore,
        \begin{align*}
            (\phi^\ast(S_2) |P_1)_1 & = (\phi^\ast(S_2) | x_0^{n_1-1} x_{g_1} \cdots x_0^{n_r-1} x_{g_r} x_0^{n_{r+1}-1})_1 \\
            & = (S_2 | x_0^{n_1-1} x_{\phi(g_1)} \cdots x_0^{n_r-1} x_{\phi(g_r)} x_0^{n_{r+1}-1})_2 = (S_2 | \phi^\sharp(P_1))_2. 
        \end{align*}
        \item It is enough to show this equality for $P_2 = x_0^{n_1-1} x_{h_1} \cdots x_0^{n_r-1} x_{h_r} x_0^{n_{r+1}-1} \in X_{G_2}^\ast$, due to linearity. We have
        \[
            \phi_\sharp(P_2) = d^{n_1+\cdots+n_{r+1}-(r+1)} \sum_{\substack{g_i\in\phi^{-1}(\{h_i\}) \\ 1\leq i\leq r}} x_0^{n_1-1} x_{g_1} \cdots x_0^{n_r-1} x_{g_r} x_0^{n_{r+1}-1}.
        \]
        Therefore,
        \begin{align*}
            (\phi_\ast(S_1) |P_2)_2 & = (\phi_\ast(S_1) | x_0^{n_1-1} x_{h_1} \cdots x_0^{n_r-1} x_{h_r} x_0^{n_{r+1}-1})_2 \\
            & = d^{n_1+\cdots+n_{r+1}-(r+1)} \sum_{\substack{g_i\in\phi^{-1}(\{h_i\}) \\ 1\leq i\leq r}} (S_1 | x_0^{n_1-1} x_{g_1} \cdots x_0^{n_r-1} x_{g_r} x_0^{n_{r+1}-1})_1 \\
            & = (S_1 | \phi^\sharp(P_2))_1.  
        \end{align*}
    \end{enumerate}
\end{proof}

\begin{corollary}
    Let $\phi : G_1 \to G_2$ be an arrow in $\mathsf{FinAb}$.
    \begin{enumerate}[label=(\alph*), leftmargin=*]
        \item \label{cor:sh_alg_morph} The map $\phi^\sharp$ (resp. $\phi_\sharp$) is an algebra homomorphism $\h_{G_1,\sh} \to \h_{G_2,\sh}$ (resp. $\h_{G_2,\sh} \to \h_{G_1,\sh}$).
        \item The map $\phi^\sharp$ (resp. $\phi_\sharp$) restricts to an algebra homomorphism $\h^1_{G_1} \to \h^1_{G_2}$ (resp. $\h^1_{G_2} \to \h^1_{G_1}$). Moreover, it is an algebra homomorphism $\h^1_{G_1,\st} \to \h^1_{G_2,\st}$ (resp. $\h^1_{G_2,\st} \to \h^1_{G_1,\st}$).
    \end{enumerate}
    \label{cor:alg_morph}
\end{corollary}
\begin{proof} \ 
    \begin{enumerate}[label=(\alph*), leftmargin=*]
        \item Let $P, Q \in \h_{G_1}$.
        To prove the statement, we will instead demonstrate that for any $w \in X_{G_2}^\ast$, we have $(\phi^\sharp(P \, \sh \, Q), w)_2 = (\phi^\sharp(P) \, \sh \, \phi^\sharp(Q), w)_2$. The equivalence of these statements allows us to proceed as follows:
        \begin{align*}
            (\phi^\sharp(P \, \sh \, Q), w)_2 & = (P \, \sh \, Q, \phi^\ast(w))_1 = \left(P \, \otimes \, Q, \widehat{\Delta}_\sh \circ \phi^\ast(w)\right)_1 \\
            & = \left(P \, \otimes \, Q, (\phi^\ast)^{\otimes 2} \circ \widehat{\Delta}_\sh(w)\right)_1 = \left(\phi^\sharp(P) \, \otimes \, \phi^\sharp(Q), \widehat{\Delta}_\sh(w)\right)_2 \\
            & = \left(\phi^\sharp(P) \, \sh \, \phi^\sharp(Q), w\right)_2,
        \end{align*}
        where the first and fourth equalities come from Lemma \ref{lem:astsharp}, the second and last ones from Lemma \ref{lem:coproductonp} and the third one from Lemma \ref{star_Hopf} \ref{star_Hopf_sh}. \newline
        The exact same arguments applies for $\phi_\sharp$.
        \item The proof mirrors the reasoning employed in \ref{cor:sh_alg_morph}, with the key difference being the application of Lemma \ref{star_Hopf} \ref{star_Hopf_st} instead of \ref{star_Hopf} \ref{star_Hopf_sh}.
    \end{enumerate}
\end{proof}

\begin{corollary}
    \label{sharp_comm}
    Let $\phi : G_1 \to G_2$ be an arrow in $\mathsf{FinAb}$. We have an equality of $\Q$-linear maps $\h_{G_1} \to \h_{G_2}$
    \begin{equation}
        \label{comm_q_up}
        \q_{G_2} \circ \phi^\sharp = \phi^\sharp \circ \q_{G_1},
    \end{equation}
    and an equality of $\Q$-linear maps $\h_{G_2} \to \h_{G_1}$
    \begin{equation}
        \label{comm_q_down}
        \q_{G_1} \circ \phi_\sharp = \phi_\sharp \circ \q_{G_2}.
    \end{equation}
\end{corollary}
\begin{proof}
    Let $P \in \h_{G_1}$. To prove the statement, we will instead demonstrate that for any $w \in X_{G_2}^\ast$, we have $(\phi^\sharp \circ \q_{G_1}(P), w)_2 = (\q_{G_2} \circ \phi^\sharp(P), w)_2$. The equivalence of these statements allows us to proceed as follows:
    \begin{align*}
            (\phi^\sharp \circ \q_{G_1}(P), w)_2 & = (\q_{G_1}(P), \phi^\ast(w))_1 = \left(P, \, \qq_{G_1}^{-1} \circ \phi^\ast(w)\right)_1 = \left(P, \, \phi^\ast \circ \qq_{G_2}^{-1}(w)\right)_1 \\
            & = \left(\phi^\sharp(P) \, , \qq_{G_2}^{-1}(w)\right)_2 = \left(\q_{G_2} \circ \phi^\sharp(P) \, , w\right)_2,
        \end{align*}
        where the first and fourth equalities come from Lemma \ref{lem:astsharp}, the second and last ones from identity \eqref{qq_and_q} and the third one from Lemma \ref{ast_comm} \ref{qphiast_comm}. Finally, one may apply similar reasoning to prove identity \eqref{comm_q_down}.
\end{proof}

    \bibliographystyle{amsalpha}
    \bibliography{main}
\end{document}